\documentclass[12pt,reqno,a4paper]{amsart}
\usepackage{amsaddr}

 \usepackage[english]{babel}
\usepackage[left=2cm,right=2cm,top=2cm,bottom=2cm]{geometry}
\usepackage{amsmath,mathrsfs,stmaryrd}
\usepackage{amsthm}
\usepackage{mathtools}
\usepackage{empheq}
\usepackage{enumerate,enumitem}
\usepackage{dsfont}
\usepackage[hidelinks]{hyperref}
\usepackage{csquotes}
\usepackage{xcolor}
\usepackage{times}

\usepackage{cancel}
\usepackage{cleveref}
\usepackage{verbatim}

\makeatletter
\@addtoreset{equation}{section}
\makeatother

\theoremstyle{plain}
\newtheorem{theorem}{Theorem}[section] 
\newtheorem{proposition}[theorem]{Proposition}
\newtheorem{corollary}[theorem]{Corollary}

\newtheorem{lemma}[theorem]{Lemma}

\theoremstyle{definition}
\newtheorem{remark}[theorem]{Remark}

\newcommand{\R}{\ensuremath{\mathbb R}} 

\newcommand{\dd}{\ensuremath{{\hspace{0.01em}\mathrm d}}}

\usepackage{amssymb,graphicx}

\definecolor{cadmiumgreen}{rgb}{0.0, 0.42, 0.24}
\definecolor{tangelo}{rgb}{0.98, 0.3, 0.0}
\definecolor{shamrockgreen}{rgb}{0.0, 0.62, 0.38}
\definecolor{darkolivegreen}{rgb}{0.33, 0.42, 0.18}
\definecolor{deepmagenta}{rgb}{0.8, 0.0, 0.8}
\definecolor{burgundy}{rgb}{0.5, 0.0, 0.13}
\definecolor{darkbyzantium}{rgb}{0.36, 0.22, 0.33}

\usepackage{lipsum}

\newcommand\blfootnote[1]{%
	\begingroup
	\renewcommand\thefootnote{}\footnote{#1}%
	\addtocounter{footnote}{-1}%
	\endgroup
}

\author[E. Gussetti, M. Sy]{Emanuela Gussetti \& Mouhamadou Sy}

\address{Bielefeld University, Germany \& AIMS Sénégal}

\title[ Properties of statistically stationary solutions to the SME ]{On Properties of Statistically Stationary Solutions to the One-Dimensional Schr\"odinger Map Equation.} 

\email{ emanuela.gussetti@uni-bielefeld.de,mouhamadou.sy@aims-senegal.org} 

\begin{document}
\date{}

	\blfootnote{ 
        \textit{Funding:} E.G.~gratefully acknowledges the financial support of the Deutsche Forschungsgemeinschaft (DFG, German Research Foundation) – SFB 1283/2 2021 – 317210226 (Project B7).  {The research of M. Sy is funded by the Alexander von Humboldt foundation under the “German Research Chair programme” financed by the Federal Ministry of Education and Research (BMBF).}\\
        \textit{ORCiD of E.G.: 0000-0002-5710-2169}
	}

	\smallskip
	
	\begin{abstract}
	We investigate further qualitative properties of statistically stationary solutions to the Schr\"odinger map equation (SME) and the Binormal Curvature Flow (BCF), continuing the work initiated by E. G., M. Hofmanová \cite{SME}. 
	Concerning the statistically stationary solutions to the SME, we show that the laws of some relevant observables (such as the space average and the energy) are absolutely continuous with respect to the Lebesgue measure, with a Gaussian decay property for the energy. We further prove that the law $\mu$ of the statistically stationary solution has dimension of at least two: this means that any compact set of Hausdorff dimension smaller than two has $\mu$-measure zero. These properties, with appropriate modifications of the norms, pass directly to the statistically stationary solutions to the BCF.
	\end{abstract}
	\maketitle
	\section{Main results}
	In this work, we continue the study started in E.~G., M.~Hofmanová \cite{SME} on the properties of statistically stationary solutions to the Schr\"odinger Map equation (SME)
	\begin{align}\label{SME_intro}
	u_t=u^0+\int_0^t u_r\times \partial_x^2 u_r \dd r\, , \quad |u_t|^2_{\mathbb{R}^3}=1\, ,
	\end{align}
	and the Binormal Curvature Flow (BCF)
	\begin{align*}
	v_t= v_0+\int_0^t \partial_x v_r \times \partial^2_x v_r \dd r\, , \quad |\partial_x v_t|^2_{\mathbb{R}^3}=1\, .
	\end{align*}
	A statistically stationary solution to the SME is a stochastic process $(z_t)_t$ on a probability space $(\Omega,\mathcal{F}, \mathbb{P})$, taking values in $H^2(\mathbb{S}^2):=H^2(D;\R^3) \cap \{g:D \rightarrow \R^3 \, \textrm{s.t.}\, |g(x)|=1\,\text{ a.e.}\, x\in D \}$, where $D$ is a bounded interval in $\mathbb{R}$, and such that
	\begin{itemize}
		\item[a.] $z(\omega)\in C([0,\infty);L^2)\cap L^2_{\mathrm{loc}}([0,\infty);H^2)$ is a solution to the SME \eqref{SME_intro} $\mathbb{P}$-a.s.,
		\item[b.] there exists a probability measure $\mu$ on $H^1(\mathbb{S}^2)$, with support in $H^2$, such that the law of $z_t$ is $\mu$ for all $t\geq 0$.
	\end{itemize}
    We postpone the discussion of the BCF to Section \ref{sec:intro_BCF}
		The existence of a statistically stationary solution to the SME is established in \cite[Theorem 1.1]{SME}. In \cite[Theorem 1.2]{SME}, several qualitative properties of the trajectories are proven: the statistically stationary solutions are non-trivial, in the sense that the map $\omega\mapsto z(\omega)$ is not constant, $\partial_x z, \partial_x^2 z\neq 0$ on a set of positive probability, and the map $(t,x)\mapsto z(t,x)$ is not constant.
        
	\subsection{On properties of the statistically stationary solutions to the SME}The main results of this paper on the SME explore further qualitative properties of the measure $\mu$ and are summarized in Theorem \ref{th:mu}.
	\begin{theorem}\label{th:mu}
		There exists a probability measure $\mu$ concentrated on $H^2(\mathbb{S}^2)$ invariant under the flow of the Schr\"odinger map equation \eqref{SME_intro} such that
		\begin{enumerate}
			\item[a.] (A Gaussian decay property): there exist constants $\sigma, C>0$ such that for all $R>0$
			\begin{align*}
				\mu(\{u\in H^2(\mathbb{S}^2): \|\partial_x u\|^2_{L^2}>R\}) \leq C \mathrm{exp}(-\sigma R^2)\, .
			\end{align*}
			\item[b.] (Small probability of small balls)  For all $\epsilon>0$, 
			\begin{align*}
				& \mu(\{u\in H^2(\mathbb{S}^2): \|u-\langle u\rangle\|^2_{L^2}<\epsilon\}) \lesssim \epsilon\, ,  \\
				&\mu(\{u\in H^2(\mathbb{S}^2): \|\partial_x u\|^2_{L^2}<\epsilon\}) \lesssim \epsilon\, ,\\
				&\mu(\{u\in H^2(\mathbb{S}^2): \|\partial^2_x u\|^2_{L^2}<\epsilon\}) \lesssim \epsilon\, , \\
				&\mu(\{u\in H^2(\mathbb{S}^2): |\langle u\rangle|^2<\epsilon\}) \lesssim \epsilon \, .
			\end{align*}
			\item[c.] The law of the conservation law $\|\partial_x u\|^2_{L^2}$ with respect to $\mu$ is continuous with respect to the Lebesgue measure on $\mathbb{R}$, i.e. for all $A\in \mathcal{B}_{\mathbb{R}}$
			\begin{align*}
				\mu(\{u\in H^2(\mathbb{S}^2): \|\partial_x u\|^2_{L^2}\in A\}) \lesssim |A|\, .
			\end{align*}
			\item[d.] The measure $\mu$ is of at least two-dimensional nature, that is, any compact set of Hausdorff dimension smaller than two has $\mu$-measure $0$.
		\end{enumerate} 
	\end{theorem}

 Our proof is based on the fluctuation-dissipation method. This was introduced by Kuksin \cite{kuksin_2003} in the context of the 2D Euler equations (see also Kuksin-Shirikyan \cite{kuksin_shirikyan_2004} on the nonlinear Schrödinger equation (NLS), and Chapter 5 of the excellent book \cite{kuksin_shirikyan_book} by the same authors). The method consists of adding to the initial PDE an appropriately scaled dissipation term and stochastic forcing, energy estimates yield the existence of a stationary measure enjoying balance relations that are uniform in the viscosity parameter. Hence, passing to the inviscid limit, one obtains an invariant measure for the initial PDE. The study of the properties of the limiting measure is a challenging question. Following Kuksin \cite{kuksin_2008} and Shirikyan \cite{Shirikyan_local_times}, one can ask about qualitative properties under the form of absolute continuity of the distributions of relevant observables with respect to the Lebesgue measure (see also \cite{sy_b_ono,latocca,sy_19, foldes_sy,holtz_sverak}). Here we prove that conservation laws of the Schrödinger map equation and a two-dimensional vector built upon them are distributed by the invariant measure in an absolutely continuous fashion with respect to the Lebesgue measure in one and two dimensions, respectively. This already rules out degenerate scenarios. It is worth mentioning some alternative approaches to the construction of invariant measures for Hamiltonian PDEs. The  Gibbs measures theory was widely applied to the NLS and similar PDEs, especially in one and two space dimensions (see e.g. \cite{lebowitz_rose_speer, Bourgain, bourgain_94,burq_thomann_tzvetkov,Tzvetkov_2008, norbs} and related works). The IID limit framework \cite{sy_19,sy_xu_2021,sy_xu_2021_2} was developed as a combination of the fluctuation-dissipation method and Gibbs-measures theory to study energy supercritical regimes.  
    
	The proof of Theorem \ref{th:mu}, as well as the proof of the existence of $\mu$ in \cite{SME}, rely on an approximation of \eqref{SME_intro} by means of a stochastic Landau-Lifshitz-Gilbert (LLG) equation on a bounded one-dimensional domain $D\subset \mathbb{R}$, given for $t\geq 0$ by
    \begin{equation} \label{LLG_nu}
    \begin{aligned}
	u^\nu_{t}&=u^\nu_{0}+\int_{0}^{t} u^\nu_r\times \partial_x^2 u^\nu_r\dd r-\nu \int_{0}^{t} u^\nu_r\times [u^\nu_r\times \partial_x^2 u^\nu_r]\dd r\\
    &\quad+ \sqrt{\nu}\sum_{j\in \mathbb{Z}}\int_{0}^{t} h_j u^\nu_r\times \circ \dd W^j_r\, ,\, \quad|u^\nu_t|_{\mathbb{R}^3}=1\, ,
	\end{aligned}
    \end{equation}
    with $u^\nu_0=u^0\in H^1(\mathbb{S}^2)$ and such that $\partial_x u^0(0)=\partial_x u^0(2\pi)=0$.
	Note that for every $\nu\in(0,1]$, the LLG equation admits an invariant measure $\mu^\nu$ for \eqref{LLG_nu} on $H^1(\mathbb{S}^2)$ endowed with the Borel $\sigma $-algebra $\mathcal{B}_{H^1(\mathbb{S}^2)}$, and an associated stationary solution $z^\nu$ (see E.~G. \cite{LLG_inv_measure}). 
	In both \cite{LLG_inv_measure} and \cite{SME}, the stochastic LLG equation is driven by a Brownian motion colored in space: in this work, we consider instead a $Q$-Wiener process and we discuss the extension to this case. 
	The stochastic LLG equation is a good candidate for approximating the SME. Indeed, it preserves the sphere, but it does not preserve the following conserved quantities for the SME,
    \begin{align}\label{eq:conserved_Q_SME}
	\langle u\rangle\, , \quad \|\partial_x u\|^2_{L^2}\, , \quad \|u-\langle u\rangle\|^2_{L^2}\, .
	\end{align}
	Since the quantities in \eqref{eq:conserved_Q_SME} are not conserved by the stochastic LLG equation and due to the scaling of the noise and the drift of the LLG, we can inherit the properties of the SME through the stationary sequence $(z^\nu)_{\nu}$. This perturbative method, known as the fluctuation-dissipation method, has been widely explored in many contexts: we refer to \cite{kuksin_shirikyan_book}. More precisely, we prove Theorem \ref{th:mu} for every $\mu^\nu$ and then pass the properties to $\mu$ via the Portmanteau theorem. 
	
	We analyze the elements and difficulties specific to the proof of Theorem \ref{th:mu}. The proof of the Gaussian decay property in Theorem \ref{th:mu} a.~is standard and can be found e.g.~in M.~Sy \cite[Theorem 6.6]{sy_b_ono}. 
	The proof of Theorem \ref{th:mu} b.~(contained in Proposition \ref{pro:first_equality_Gamma} and Proposition \ref{pro:ineq}), consists in proving first the inequality
	\begin{align}\label{eq:intro_small_b_u_minus_av}
		\mu^\nu(\{u\in H^2(\mathbb{S}^2): \|u-\langle u\rangle\|^2_{L^2}<\epsilon\}) \lesssim \epsilon\, 
	\end{align}
	for all $\nu\in (0,1]$ and for all $\epsilon>0$. This implies a bound on the first and second derivatives from the Poincaré inequality. We follow the strategy of Shirikyan \cite{Shirikyan_local_times}: a combination of It\^o's formula and local times techniques applied to the approximating stationary sequence $(z^\nu)_{\nu}$. Then, again for every fixed $\nu$, we show that
	\begin{align*}
	\mu^\nu(\{u\in H^2(\mathbb{S}^2): |\langle u\rangle|^2<\epsilon\}) \lesssim \epsilon\, ,  
	\end{align*}
	for all $\nu\in (0,1]$ and for all $\epsilon>0$. The proof is a combination of Corollary \ref{cor:null_measure_soace_average}, \eqref{eq:intro_small_b_u_minus_av} and Proposition \ref{pro:space_av_small}.
	Theorem \ref{th:mu} c.~is a combination of the strategy of Shirikyan \cite{Shirikyan_local_times}, with Theorem \ref{th:mu} b.
	 The proof of the fact that the measure is at least two dimensional, in Theorem \ref{th:mu} d., is based on a Krylov estimate (see Proposition \ref{pro:krylov}) and the main difficulty is to establish a lower bound on a determinant of a matrix $2\times 2$ depending on the solution $z^\nu$. The proof is inspired by the work of Kuksin and Shirikyan \cite[Section 5.2.3]{kuksin_shirikyan_book}, as well as by the work of S. \cite{sy_b_ono}. 
	 
	 Besides its general structure, many arguments are technical and specific to the problem under investigation. We enumerate the main difficulties and contributions of this work:
     \begin{enumerate}
         \item[a.] \textit{The structure of the noise:} usually, the fluctuation-dissipation method consists in perturbing the limit equation by a dissipation operator and an additive noise. In our case, the noise is valued on the tangent plane orthogonal to the solution, and in particular, it is a multiplicative noise. This structure adds some challenges in the estimate of the lower bounds of the quadratic variation term (see e.g.~the proof of Proposition \ref{pro:ineq} and Proposition \ref{pro:inequ_grad_leb}).
         \item[b.] \textit{The structure of the equation:}
         Despite the deep link between the SME and the one-dimensional cubic focusing non-linear Schr\"odinger equation (NLS), which has infinitely many conserved quantities (see Zakharov, Shabat \cite{Zakharov_Shabat}), through the Hashimoto transformation, it is unclear how to obtain a similar integrable structure for the former equation. This makes its analysis much more difficult than in the one-dimensional NLS. In particular, only two coercive conservation laws are known for the SME (seepoint c. below).
         \item[c.] \textit{The number and type of available conservation laws:} The more conservation laws independent from each other, the more dimensions can be expected for the limit measure $\mu$. One can employ the famous Krylov estimate to establish such properties.  In our case, we only have two independent conservation laws. As a matter of fact, one needs to show that the derivatives of these functionals are linearly independent. In the case of the 2D Euler or SQG-type equations, this property can be prescribed by carefully choosing the integrands of the Casimir functionals (see e.g. Kuksin and Shirikyan \cite[Section 5.2.3]{kuksin_shirikyan_book}, F\"oldes and S. \cite{foldessy2023almost}). When such a prescription is not available, it can be difficult to prove the needed independence ( see e.g. the case of the Benjamin-Ono equation \cite{sy_b_ono}, where it is possible to exploit only two of the conservation laws, despite the existence of an infinite number of these.) One challenge in the present work is to establish the two-dimensionality property of the measure by proving the independence of the available two conservation laws.
          
     \end{enumerate}
     \begin{remark}
         The conclusions of Theorem \ref{th:mu} hold for every fixed invariant measure $\mu^\nu$ of the stochastic LLG equation and are also new in this case.
     \end{remark}
     \begin{remark}
         In \cite{SME}, we can prove that the set of spatially nontrivial dynamics is of positive probability, namely.
         \begin{align*}
             \mathbb{P}(\{u\in H^2(\mathbb{S}^2): \|\partial_x u\|^2_{L^2}>0)\})>0\, .
         \end{align*}
         A straightforward consequence of Theorem \ref{th:mu} is that 
        \begin{align*}
             \mathbb{P}(\{u\in H^2(\mathbb{S}^2): \|\partial_x u\|^2_{L^2}>0)\})=1\, .
         \end{align*}
     \end{remark}
     
    \subsection{On properties of the statistically stationary solutions to the Binormal Curvature Flow.}\label{sec:intro_BCF}
	By applying to each trajectory $z(\omega)$ of the SME \eqref{SME_intro} the transformation 
	\begin{align*}
	v_t(\omega):=f(z_t(\omega))\, ,\quad\quad\mathrm{where}\quad f(g):=\int_{0}^{\cdot} g(x)\dd x\, ,
	\end{align*}
	we obtain a solution $v(\omega)$ to the binormal curvature flow (BCF), namely, its time evolution is described for all $t\geq 0$
	\begin{align}\label{eq:BCF_intro}
	v_t(\omega)= v_0(\omega)+\int_0^t \partial_x v_r (\omega)\times \partial^2_x v_r (\omega)\dd r\, ,\quad |\partial_x v|_{\mathbb{R}^3}^2=1\,  .
	\end{align}
	As stated in \cite[Theorem 1.6]{SME}, given a statistical stationary solution $z$ to the SME on $H^2(\mathbb{S}^2)$, the stochastic process $f(z)$ is a statistically stationary solution to \eqref{eq:BCF_intro} on $H^3$. We denote the law of the statistically stationary solution $f(z)$ by $\tilde{\mu}$. Note that $\mathbb{P}$-a.s.~$\|\partial^{k+1}_x v\|^2_{L^2}= \|\partial^{k}_x z\|^2_{L^2}$ for all $k\in \mathbb{N}$ and the equalities hold
	\begin{align}\label{eq:rel_SME_BCF}
	 \frac{v(2\pi)}{2\pi}=\langle z\rangle\, , \quad v(0)=0\, .
	\end{align}
	As a consequence of Theorem \ref{th:mu}, we inherit some properties of the statistically solution $\tilde{\mu}$ to the BCF from the statistically stationary solution $\mu$  to the SME. The assertions in Theorem \ref{th:mu_tilde_BCF} a., b., c., d.~are straightforward consequences of the definition and the relations in \eqref{eq:rel_SME_BCF} (see the proof in Section \ref{sec:proof_BCF}).
	\begin{theorem}\label{th:mu_tilde_BCF}
	There exists a probability measure $\tilde{\mu}$ on $H^1(\mathbb{S}^2)$, concentrated on $H^3$, invariant under the flow of the Binormal Curvature Flow \eqref{eq:BCF_intro} such that
	\begin{enumerate}
		\item[a.] (a Gaussian decay property): there exist the constants $\sigma, C>0$ such that for all $R>0$
		\begin{align*}
		\tilde{\mu}(\{u\in H^3: \|\partial_x u\|^2_{L^2}>R\}) \leq C \mathrm{exp}(-\sigma R^2)\, .
		\end{align*}
		\item[b.] (small probability of small balls)  for all $\epsilon>0$, 
		\begin{align*}
		\quad \quad\tilde{\mu}(\{u\in H^3: \frac{|u(2\pi)|^2}{4\pi^2}&<\epsilon\}) \lesssim \epsilon\, , \quad \tilde{\mu}(\{u\in H^3: 1-\frac{|u(2\pi)|^2}{4\pi^2}<\epsilon\}) \lesssim \epsilon\, ,  \\
		&\tilde{\mu}(\{u\in H^3: \|\partial^2_x u\|^2_{L^2}<\epsilon\}) \lesssim \epsilon\, .
		\end{align*}
		\item[c.] the law of the conservation law $\|\partial_x u\|^2_{L^2}$ with respect to $\tilde{\mu}$ are continuous with respect to the Lebesgue measure on $\mathbb{R}$, i.e., for all $A\in \mathcal{B}_{\mathbb{R}}$
		\begin{align*}
		\tilde{\mu}(\{u\in H^3: \|\partial^2_x u\|^2_{L^2}\in A\}) \lesssim |A|\, .
		\end{align*}
		\item[d.] The measure $\tilde{\mu}$ is of at least two-dimensional nature, i.e., any compact set of Hausdorff dimension smaller than two has $\tilde{\mu}$-measure $0$.
	\end{enumerate} 
	\end{theorem}
	The SME and the BCF are related to the cubic focusing non-linear Schr\"odinger equation. We refer the interested reader to the works of V.~Banica and L.~Vega \cite{banica_vega_1,banica_vega_2,banica_vega_3,banica_vega_4} or to the short remark in \cite[Section 1.3.1]{SME}.

\subsection{Overview on the literature.}
\textit{Stochastic LLG equation:} concerning the well posedness of the stochastic equation in one dimension, we mention: the martingale approach via the classical Stratonovich calculus (introduced in Z.~Brze\'zniak, B.~Goldys, Jegeraj \cite{brzezniak_LDP}) and the rough paths approach (explored in E.~G., A.~Hocquet \cite{LLG1D}, K.~Fahim, E.~Hausenblas, D.~Mukherjee \cite{fahim_h},  E.~G.~\cite{CLT} and A.~Hocquet, A.~Neamţu \cite{hocquet_neamtu}). We also mention the monograph from B.~Guo and X.~Pu \cite{stochastic_LL}. The solutions obtained using the two approaches coincide up to a set of measure zero. This follows from the fact that the Stratonovich stochastic integral and its rough paths equivalent coincide up to a set of measure zero (see, e.g.,~\cite[Theorem 5.14]{FrizHairer}).
The long-time behaviour of a finite-dimensional ensemble of particles is analyzed in M.~Neklyudov, A.~Prohl \cite{neklyudov2013role}. In E.G.~\cite{LLG_inv_measure}, the existence of an invariant measure is proved.

\textit{Binormal curvature flow and vortex filament equation.} Some relevant probabilistic approaches to the study of the vortex filament equation are considered in M.~Gubinelli, F.~Flandoli \cite{flandoli_gubinelli_1,flandoli_gubinelli_2}, H.~Bessaih, M.~Gubinelli, F.~Russo \cite{bess_gub_russo}, and Z.~Brze\'zniak, M.~Gubinelli, M.~Neklyudov \cite{brz_gub_nek} (and references therein). We mention \cite{banica_luca_tzv_vega} for another recent probabilistic result on the binormal curvature flow and \cite{da_rios} for the relation between the BCF and the vortex filament equation.

	\section{Notations and setting}
	\subsection{Some notations}
	Let $a,b\in \R^3$ be two vectors. We denote their inner product by $a\cdot b$, and by $|\cdot|$ the norm inherited from it (we will not distinguish between the different dimensions, it will be clear from the context). The cross product between $a\equiv(a_1,a_2,a_3)$ and $b\equiv(b_1,b_2,b_3)$ is defined as
	$a\times b:=(a_2b_3-a_3b_2,a_3b_1-a_1b_3,a_1b_2-a_2b_1)$.
	By $\mathbb{S}^2:=\{a\in\mathbb{R}^3:|a|_{\mathbb{R}^3}=1 \}$ we denote the unit sphere in three dimensions.
	Let $E$ be a Banach space. For $T>0$, we denote the space of continuous functions defined on $[0,T]$ with values in $E$ by $C([0,T];E)$. We denote by $C_{\mathrm{loc}}([0,+\infty);E)$ the space of maps $f\in C([0,T];E)$ for all $T>0$. For $\alpha \in [0,1]$, $C^\alpha([0,T];E)$ is the space of $\alpha$-H\"older continuous functions from $[0,T]$ with values in $E$.\\
	
	We define $C_{\mathrm{w}}([0,T]; E)$ as the space of continuous functions $f:[0,T]\rightarrow E$ that are continuous with respect to the weak topology, i.e., such that the scalar functions $t\mapsto \langle g^*, f(t,\cdot) \rangle_{E^*,E} $ belong to $C([0,T];\mathbb{R})$ for any $g^*\in E^*$. A sequence $(f_n)_n$ is said to converge to $f$ in $C_{\mathrm{w}}([0,T]; E)$ if
	\begin{align*}
	\lim_{n\rightarrow 0}\sup_{t\in [0,T]}|\langle g^*,f_n-f\rangle_{E^*,E}|=0 \quad \forall g^*\in E^*\, .
	\end{align*}
    We denote by $\mathfrak{p}(E)$ the space of Borel probability measures on $E$. We denote by $B_b(E)$ the space of real-valued bounded and Borel-measurable functions and by $C_b(E)$ the space of real-valued continuous and bounded functions.
    
	Let $D\subset\mathbb{R}$ be an open, bounded interval. We denote the space of natural numbers by $\mathbb{N}$  and define $\mathbb{N}_0:=\mathbb{N}\cup\{0\}$.
	For each $n\in\mathbb{N}$, we work with the standard Lebesgue spaces $L^p:=L^p(D ;\R^n)$ for $p\in[1,+\infty]$, equipped with the usual norm $\|\cdot\|_{L^p}$. Additionally, we consider the classical Sobolev spaces $W^{k,p}:=W^{k,p}(D;\R ^n)$ for $k\in\mathbb{N}_0$, endowed with the norm $\|\cdot\|_{W^{k,q}}$.
	For $p=2$, we write $H^k:=W^{k,2}(D ;\R ^n)$. 
	We are also interested in functions that take values in the unit sphere $\mathbb{S}^2\subset\R^3$: we define the Sobolev space of sphere-valued functions as
	\begin{equation*}
	H^k(\mathbb{S}^2):=H^k(D;\R^3)\cap \{g:D \rightarrow \R^3 \, \textrm{s.t.}\, |g(x)|=1\,\text{a.e.}\, x\in D \}\, ,
	\end{equation*}
	for $k\in\mathbb{N}_0$. 
    
	Finally, we will denote by $L^p(W^{k,q}):=L^p([0,T];W^{k,q}(D ;\R ^n))$. We write $C^k_0(D)$ for the space of real-valued functions, with compact support on $D$ and $k$-times continuously differentiable. Let $(\Omega,\mathcal{F},\mathbb{P})$ be a probability space: we denote by $\mathcal{L}^p(\Omega;E)$ the space of $p$-integrable $E$-valued random variables with respect to the probability measure $\mathbb{P}$.
	
	Let $X\equiv(X_t)_{t\geq 0}$ be a $E$-valued stochastic process. We say that $X$ is stationary provided the laws
	\begin{align*}
	\mathcal{L}(X_{t_1},\, X_{t_2}, \dots, X_{t_n})\, ,\quad \mathcal{L}(X_{t_1+\tau},X_{t_2+\tau},\dots, X_{t_n+\tau})
	\end{align*}
	coincide on $E^{\times n}$ for all $\tau>0$  and for all $t_1,\cdots,t_n\in [0,+\infty)$. 
	
	\subsection{An infinite-dimensional noise with values on the sphere.}\label{sec:noise}
	To simplify the computations, we fix $D=[0, 2\pi]$, but the results in this paper hold for any bounded domain.
	We introduce an orthonormal basis $\{e_n\}_{n\in \mathbb{Z}}$ of $L^2(D;\mathbb{R} )$, given by 
	\begin{align}\label{eq:ONB}
	e_n(x):=\left\{
	\begin{array}{cc}
	\frac{\sin (n x)}{\sqrt{\pi}} & \text{ if } n>0   \\
	\frac{1}{\sqrt{2\pi}}							& \text{ if } n=0 \\
	\frac{\cos(n x)}{\sqrt{\pi}} & \text{ otherwise.}
	\end{array}
	\right.
	\end{align}
	We now construct a $Q$-Wiener process with values in the sphere. The construction is not new, and we give a small modification of the consideration in \cite[Remark 2.8]{brz_gol_li}.
	
	Let $(\Omega,\mathcal{F}, (\mathcal{F}_t)_t,\mathbb{P})$ be a filtered probability space. 
	\begin{enumerate}
		\item[1.] Fix $j\in \mathbb{N}$ and let $W^j\equiv (W^{j,1}, W^{j,2}, W^{j,3})$ be an $\mathbb{R}^3$-valued $(\mathcal{F}_t)_t$-Brownian motion with $\mathbb{R}$-valued $(\mathcal{F}_t)_t$-Brownian motion $W^{j,i}$ independent of $W^{j,k}$ for all $i\neq k\in \{1,2,3\}$.
		Assume that each $W^j$ is independent of $W^i$ for all $i\neq j\in \mathbb{Z}$.
		\item[2.]  Assume that 
		\begin{align}
			\sum_{j\in \mathbb{Z}} \lambda_j^2\|e_j\|^2_{W^{1,\infty}}= \sum_{j\in \mathbb{Z}} \lambda_j^2 (1+j^2)<+\infty\, .
		\end{align}
		for a sequence $(\lambda_j)_{j\in\mathbb{Z}}\subset\mathbb{R}$. We further assume, without loss of generality, that $\lambda_j=\lambda_{-j}$ and $(\lambda_j)_{j\geq 0}$ is monotonically decreasing.
	\end{enumerate}
Following \cite[Remark 2.8]{brz_gol_li}, we can construct a $Q$-Wiener process on $L^2$ 
\begin{align}\label{eq:trace_class_noise}
	\sum_j \lambda_j e_jW^j\, ,
\end{align}
with a non-negative, symmetric, trace-class covariance operator $Q$ determined uniquely by
\begin{align*}
	\langle Q v, u\rangle=\sum_j (\lambda_j e_j , u)_{L^2} (\lambda_j e_j , v)_{L^2}\, ,\quad\quad \forall u,v\in L^2(D; \mathbb{R})\, .
\end{align*}

\section{Well posedness and existence of an invariant measure for the stochastic LLG with trace class noise}
This small section is devoted to the extension of the results in \cite{LLG_inv_measure} to the setting of the trace-class noise introduced before. The stochastic Landau-Lifschitz-Gilbert equation with trace-class noise reads
\begin{align}\label{LLG}
u_t=u_0+\int_0^t \left[-u_r\times \partial_x^2 u_r+ u_r\times [u_r\times \partial_x^2 u_r] -  \sum_j \lambda_j^2 e_j^2  u_r \right]\dd r +\sum_j \int_0^t \lambda_j e_j u_r\times \dd W^j_r\, ,
\end{align}
where $u_0=u^0\in H^1(\mathbb{S}^2)$ and with the notations of Section \ref{sec:noise}.
\subsection{Existence and uniqueness of a solution to \eqref{LLG}.} There are already existing results on well-posedness of the stochastic LLG in one dimension with an infinite-dimensional noise: the proofs fall in the framework of the classical It\^o-Stratonovich calculus. We mention for instance the work of Brzézniak, Goldys, Jegeraj \cite{brzezniak_LDP}. 

The rough-path framework developed in Gussetti, Hocquet \cite{LLG1D} applies straightforwardly also to the trace-class noise \eqref{eq:trace_class_noise}, by noticing that a $Q$-Wiener process can be lifted to a geometric rough path as in
\cite[Exercise 3.16]{FrizHairer}. 
The rough path lift of the trace class noise \eqref{eq:trace_class_noise} can then be reshaped into a bounded rough driver, see \cite[Example 2.5]{LLG1D}.

Thus, all the results in \cite{LLG1D} follow. In particular, the solution is continuous with respect to the initial condition, i.e., for $u, v$ unique solutions to \eqref{LLG} started in $u^0, v^0\in H^1(\mathbb{S}^2)$ respectively, it holds
\begin{align}\label{eq:cont_init_data_1}
\sup_{t\in[0,T]}\|u_t-v_t\|_{H^1}\leq C \|u^0-v^0\|_{H^1}\, ,\quad \mathbb{P}-\mathrm{a.s.}
\end{align}
for a constant $C>0$ depending on $\exp(T), |D|, u^0, v^0, \mathbf{W}$ (here with $\mathbf{W}$ we denote the rough paths lift of the trace-class noise).

\subsection{Existence of an invariant measure.}
Consider the Markov semigroup of linear operators $(P_t)_t$, defined for all $\phi\in B_b(H^1(\mathbb{S}^2))$ by,
\begin{align}\label{eq:transition_semigroup}
P_t\phi (x):=\mathbb{E}\left[\phi(u^x_t)\right]\, ,
\end{align}
where $u^x_t$ is the solution to \eqref{LLG} at time $t$ with initial condition $x\in H^1(\mathbb{S}^2)$. 

The proof of the existence of an invariant measure on $H^1(\mathbb{S}^2)$ relies on the Krylov-Bogolioubov theorem (see e.g., \cite{martina_book}), which we recall: 
\begin{theorem}\label{teo:Krylov_B}
	Let $E$ be a Polish space and let $(P_t)_t$ be a Markov semigroup with the Feller property on $C_b(E)$. Consider a random variable $u^0$ with values in $E$ with law $\mu$ and denote by $\mu^{u^0}_{t}:=P_t^*\mu$.
	
	Assume that there exists a divergent monotone increasing sequence of times $(t_n)_n$ so that the sequence of probability measures $(\mu_{t_n})_n\subset \mathcal{P}(E)$, defined for all $A\in \mathcal{B}_{E}$ by
	\begin{align*}
	\mu_{t_n}(A):=\frac{1}{t_n}\int_{0}^{t_n}\mu^{u^0}_s(A)  \dd s\, ,
	\end{align*}
	is tight. Then there exists at least one invariant measure for $(P_t)_t$. 
	
\end{theorem}
\textit{The Feller property:} The Markov semigroup \eqref{eq:transition_semigroup} has the Feller property, i.e., $P_t:C_b(H^1(\mathbb{S}^2))\rightarrow C_b(H^1(\mathbb{S}^2))$ for all $t>0$. Indeed, from the continuity in \eqref{eq:cont_init_data_1} and the dominated convergence theorem,
\begin{align}\label{eq:continuity_phi}
\lim_{n\rightarrow +\infty} (P_t\phi)(x^n)=\lim_{n\rightarrow +\infty} \mathbb{E}\left[\phi(u_t^{x^n})\right]=\mathbb{E}\left[\lim_{n\rightarrow +\infty}\phi(u_t^{x^n})\right]=\mathbb{E}\left[\phi(u_t^{x})\right]=(P_t\phi)(x)\, ,
\end{align}
for all $\phi \in C_b(H^1(\mathbb{S}^2))$. We refer to \cite[Theorem 1.4]{LLG_inv_measure} for more details.

\textit{Tightness of the sequence $(\mu_{t_n})_{t_n}$ on $H^1(\mathbb{S}^2)$.} The proof of this step is analogous to \cite[Lemma 4.13]{LLG_inv_measure}. The only change in the proof comes from the current choice of the noise, cf. \cite[Lemma 4.7]{LLG_inv_measure}. 
\begin{lemma}\label{lemma:new_growth_LLG}
For all $t\geq 0$, the equality holds
		\begin{equation}\label{eq:unif_bound_time}
	\begin{aligned}
	\sup_{r\in[0,t]}\mathbb{E}\left[\|\partial_x u_r\|^2_{L^2}\right]+2\int_{0}^{t}\mathbb{E}\left[\|u_r\times\partial^2_x u_r\|^2_{L^2}\right]\dd r = \mathbb{E}\left[\|\partial_x u^0\|_{L^2}^2\right]+ \frac{2}{\pi}t \sum_j j^2\lambda^2_j\, .
	\end{aligned}
	\end{equation}
\end{lemma}
\begin{proof}
By applying the multidimensional It\^o formula (see e.g. \cite[Theorem 2.4.1]{martina_book}) to $F(\partial_x u):=\|\partial_x u\|^2_{L^2}$. Since $F'(\partial_x  u)= 2 \partial_x  u\in L^2$ and $F''(\partial_x u)=2I\in \mathbb{R}^{3\times 3}$, we have
\begin{align*}
\| \partial_x u_t\|_{L^2}^2-\| \partial_x u_0\|_{L^2}^2&=2\int_{0}^{t} \left(\partial_x A(u_r) , \partial_x u_r\right)_{L^2} \dd r\\
&\quad+\frac{1}{2}\sum_j \int_{0}^{t}\int_D \mathrm{tr}\left(\left(\Gamma(\partial_x(\lambda_j e_j u_r))^T2 I \Gamma(\partial_x(\lambda_j e_j u_r))\right) \right)\dd x\dd r\\
&\quad+2\sum_j \int_{0}^{t} \left(\partial_x u_r,\lambda_j \partial_x(e_j u_r)\times \dd W^j_r\right)_{L^2}\, ,
\end{align*}
where $\Gamma$ is defined by
	\begin{align}\label{eq:def_B}
	\Gamma(v)= 
	\left[ {\begin{array}{ccc}
		0& -v_3  & v_2\\
		v_3 & 0 &-v_1\\
		-v_2	& 	v_1	&0 \\
		\end{array} } \right]\, ,
	\end{align}
for all $v\equiv(v_1,v_2,v_3)\in \mathbb{R}^3$, and the drift is given by
\begin{align}\label{eq:A_drift}
A(u_r) = -u_r\times \partial_x^2 u_r+ u_r\times [u_r\times \partial_x^2 u_r] -  \sum_j \lambda_j^2 e_j^2  u_r\, .
\end{align}
For all $v\equiv(v_1,v_2,v_3)\in \mathbb{R}^3$, the equality holds
\begin{align*}
(\Gamma(v) I \Gamma(v)^T)^{i,i}=v_j^2+v_k^2 \, ,
\end{align*}
for $j\neq k\neq i$ and $i,j,k\in \{1,2,3\}$. Hence,
\begin{align*}
\mathrm{tr}(\Gamma(v) I \Gamma(v)^T)=2|v|^2\, ,
\end{align*}
and in particular, recalling that $|u_r|_{\mathbb{R}^3}=1$ and $\partial_x u_r\cdot u_r =0$ for a.e. $x\in D$ and $\mathbb{P}$-a.s., 
\begin{align*}
\mathrm{tr}\left(\left(\Gamma(\partial_x(\lambda_j e_j u_r))^T2 I \Gamma(\partial_x(\lambda_j e_j u_r))\right) \right)&=4|\partial_x(\lambda_j e_j u_r)|^2\\
&=4\lambda^2_j |\partial_x e_j|^2+4\lambda^2_j |e_j|^2 |\partial_x u_r|^2\, .
\end{align*}
By rewriting the drift as
\begin{align*}
	\left(\partial_x A(u_r) , \partial_x u_r\right)&=- \|u_r\times \partial_x^2 u_r\|^2_{L^2}-\sum_j \left(2\lambda^2_j e_j\partial_x e_j u_r+\lambda^2_j e_j^2 \partial_x u_r, \partial_x u_r\right)_{L^2}\\
	&= - \|u_r\times \partial_x^2 u_r\|^2_{L^2}-\sum_j   \lambda^2_j \|e_j \partial_x u_r\|_{L^2}^2\, ,
\end{align*}
the equality holds $\mathbb{P}$-a.s.
\begin{align*}
\| \partial_x u_t\|_{L^2}^2-\| \partial_x u_0\|_{L^2}^2&=-2\int_{0}^{t} \|u_r\times \partial_x^2 u_r\|^2_{L^2} \dd r+2t \sum_j \|\partial_x e_j\|^2_{L^2}\lambda^2_j\\
&\quad+2\sum_j \int_{0}^{t} \left(\partial_x u_r,\lambda_j \partial_x(e_j u_r)\times \dd W^j_r\right)_{L^2}\, .
\end{align*}
By taking expectations and noting that the stochastic integral is a martingale with zero mean, and since
\begin{align}\label{eq:ONB}
	\partial_x e_n(x):=\left\{
	\begin{array}{cc}
	\frac{n\cos (n x)}{\sqrt{\pi}} & \text{ if } n>0   \\
	0							& \text{ if } n=0 \\
	\frac{-n\sin(n x)}{\sqrt{\pi}} & \text{ otherwise},
	\end{array}
	\right.
	\end{align}
the claim follows.
\end{proof}
Changing \cite[Lemma 4.7]{LLG_inv_measure} with Lemma \ref{lemma:new_growth_LLG}, the tightness follows from the arguments in \cite{LLG_inv_measure}.

\section{Known results on the statistically stationary solutions to the SME}

With considerations analogous to those above, the results in G., Hofmanová \cite{SME} extend directly to the framework of infinite-dimensional trace class noise as constructed in Section \ref{sec:noise}. We recall the main results used throughout this work.

\begin{lemma}\label{lemma:unif_bounds_nu}
\begin{itemize}
\item[1.] \cite[Corollary 1.4]{SME} Assume that there exists $j\in \mathbb{Z}\setminus \{0\}$ such that $\lambda_j\neq 0$. Let $z^\nu$ be a stationary solution to \eqref{LLG_nu}. For every fixed $\nu\in(0,1]$, 
\begin{align*}
\mathbb{P}(\{\omega \in \Omega: \|\partial_x z^\nu_t(\omega)\|^2_{L^2}>0\quad \forall t\geq 0\})=1\, .
\end{align*}
\item[2.] For all $t \geq 0$ and $z^\nu$ be a stationary solution to \eqref{LLG_nu}, 
\begin{align}\label{eq:equality_u_cross_laplacian}
\mathbb{E}[\|z^\nu_t\times \partial_x^2 z^\nu_t\|^2_{L^2}]=\frac{2}{\pi} \sum_j j^2\lambda^2_j =:L^2\, ,\quad \mathbb{E}[\|\partial_x^2 z^\nu_t\|_{L^2}]\lesssim L^2 \, .
\end{align}
\end{itemize}

\end{lemma}

\section{Proofs}
\subsection{Theorem \ref{th:mu} a.}
\begin{proposition}
There exist constants $\sigma, C>0$ such that for all $R>0$
\begin{align*}
\mu(\{u\in H^2(\mathbb{S}^2): \|\partial_x u\|^2_{L^2}>R)\}) \leq C \mathrm{exp}(-\sigma R^2)\, .
\end{align*}
\end{proposition}
\begin{proof} The proof is a straightforward adaptation of \cite[Theorem 6.6]{sy_b_ono}, by observing that from \cite[Lemma 3.5]{SME}, it holds
\begin{align*}
\mathbb{E}[\|\partial_x z^\nu\|^p_{L^2}]\lesssim p^p L^p\, ,
\end{align*}
where $p\in \mathbb{N}\setminus \{0\}$ and the constant is independent of $p, \nu$.
\end{proof}
\subsection{Theorem \ref{th:mu} b., c.}
	\begin{proposition}\label{pro:first_equality_Gamma}
		For all $g\in C^2(\mathbb{R};\mathbb{R})$ and for all $t\geq 0$, the equality holds
		\begin{align}\label{eq:ugly_formula}
		&\quad\int_{\Gamma}\mathbb{E}\left[\mathds{1}_{(a,\infty)}(f_t) [-g'(f_t)\langle z^\nu_t|\partial_x z^\nu_t|^2\rangle \cdot \langle z^\nu_t\rangle+g'(f_t)\sum_j \lambda_j^2\left[\langle e_j^2 z^\nu_t \rangle\cdot \langle z^\nu_t\rangle-|\langle  e_j z^\nu_t\rangle |^2 \right]]\right] \dd a\\
		&\quad\quad+\int_{\Gamma}\sum_j \lambda_j^2|D|\mathbb{E}\left[\mathds{1}_{(a,\infty)}(f_t) g''(f_t)\left|\langle e_j z^\nu_t\rangle\times\langle z^\nu_t\rangle\right|^2\right] \dd a\\
		&\quad\quad+\sum_j \lambda_j^2|D| \mathbb{E}\left[\mathds{1}_\Gamma (g(f_t))\left|\langle e_j z^\nu_t\rangle\times\langle z^\nu_t\rangle\right|^2g'\left(f_t\right)^2\right]=0\, ,
		\end{align}
		where $f_t:=\|z^\nu_t-\langle z^\nu_t\rangle\|^2_{L^2}$.
	\end{proposition}
\begin{proof}
For all $t\geq 0$, we rewrite the norm $\|z^\nu_t-\langle z^\nu_t\rangle\|^2_{L^2}$ as
\begin{align*}
\|z^\nu_t-\langle z^\nu_t\rangle\|^2_{L^2}=\langle z^\nu_t-\langle z^\nu_t\rangle , z^\nu_t-\langle z^\nu_t\rangle \rangle_{L^2}&=|D|-2|D|\langle z^\nu_t\rangle \cdot \langle z^\nu_t\rangle +|D| \langle z^\nu_t\rangle \cdot \langle z^\nu_t\rangle \\
&=|D|-|D| |\langle z^\nu_t\rangle |^2\, .
\end{align*}
Thus the evolution of $\|z^\nu_t-\langle z^\nu_t\rangle\|^2_{L^2}$ reduces that of $|\langle z^\nu_t\rangle |^2$. From It\^o's formula and from the stochastic Fubini's theorem (see e.g. \cite[Theorem IV.65]{protter}), the equality holds
\begin{equation}\label{eq:equation_space_average_squared}
\begin{aligned}
|\langle z^\nu_t\rangle|^2 -|\langle z^\nu_0\rangle|^2&=2\nu\int_{0}^{t}\langle z^\nu_r|\partial_x z^\nu_r|^2\rangle \cdot \langle z^\nu_r\rangle \dd r-2\nu\sum_j \lambda_j^2 \int_{0}^{t}\langle e_j^2 z^\nu_r \rangle\cdot \langle z^\nu_r\rangle \dd r\\
&\quad+2\nu \sum_j \lambda_j^2\int_{0}^{t}|\langle  e_j z^\nu_r\rangle |^2\dd r+2 \sum_j \lambda_j \int_{0}^{t}   \langle z^\nu_r\rangle\cdot\langle \sqrt{\nu} e_j z^\nu_r\times \dd W^j_r\rangle \, ,
\end{aligned}
\end{equation}
for all $ t\geq 0$. Since $e_0=1$ and by using the fact that $a\times b\cdot a=0$ for all $a,b\in \mathbb{R}^3$
\begin{align*}
-2\nu\sum_j \lambda_j^2 \int_{0}^{t}\langle e_j z^\nu_r \rangle\cdot \langle z^\nu_r\rangle \dd r+2\nu \sum_j \lambda_j^2\int_{0}^{t}|\langle  e_j z^\nu_r\rangle |^2\dd r+2 \sum_j \lambda_j \int_{0}^{t}   \langle z^\nu_r\rangle\cdot\langle \sqrt{\nu} e_j z^\nu_r\times \dd W^j_r\rangle\\
=-2\nu\sum_{j\neq 0} \lambda_j^2 \int_{0}^{t}\langle e_j z^\nu_r \rangle\cdot \langle z^\nu_r\rangle \dd r+2\nu \sum_{j\neq 0} \lambda_j^2\int_{0}^{t}|\langle  e_j z^\nu_r\rangle |^2\dd r+2 \sum_{j\neq 0} \lambda_j \int_{0}^{t}   \langle z^\nu_r\rangle\cdot\langle \sqrt{\nu} e_j z^\nu_r\times \dd W^j_r\rangle\, .
\end{align*}
The evolution of $\|z^\nu_t-\langle z^\nu_t\rangle\|^2_{L^2}$ is given by
\begin{align*}
\|z^\nu_t-\langle z^\nu_t\rangle\|^2_{L^2}&-\|z^\nu_0-\langle z^\nu_0\rangle\|^2_{L^2}=-2\nu|D|\int_{0}^{t}\langle z^\nu_r|\partial_x z^\nu_r|^2\rangle \cdot \langle z^\nu_r\rangle \dd r+2\nu|D|\sum_{j\neq 0} \lambda_j^2 \int_{0}^{t}\langle e_j^2 z^\nu_r \rangle\cdot \langle z^\nu_r\rangle \dd r\\
&\quad\quad\quad\quad-2\nu |D|\sum_{j\neq 0} \lambda_j^2\int_{0}^{t}|\langle  e_j z^\nu_r\rangle |^2\dd r
-2 |D|\sum_{j\neq 0} \lambda_j \int_{0}^{t}   \langle z^\nu_r\rangle\cdot\langle \sqrt{\nu} e_j z^\nu_r\times \dd W^j_r\rangle \, .
\end{align*}
Let $g\in C^2(\mathbb{R};\mathbb{R})$. We describe the evolution of $g(f_t)$ with $f_t:=\|z^\nu_t-\langle z^\nu_t\rangle\|^2_{L^2}$ by It\^o's formula as
\begin{align*}
g(f_t)=g(f_0)+\int_0^t g'(f_r) \dd f_r +\frac{1}{2}\int_0^t g''(f_r)\dd\langle  \langle f_r \rangle \rangle\, .
\end{align*}
We compute the quadratic variation $\langle \langle f\rangle\rangle_t$
\begin{align*}
\langle \langle f\rangle\rangle_t&=\langle \langle 2\sum_{j\neq 0} \lambda_j|D| \int_{0}^{\cdot} \langle z^\nu_r\rangle\cdot\langle \sqrt{\nu} e_j z^\nu_r\rangle\times \dd W^j_r\rangle\rangle_t\\
&=\langle \langle 2\sum_{j\neq 0} \lambda_j |D| \int_{0}^{\cdot} \langle \sqrt{\nu} e_j z^\nu_r\rangle\times\langle z^\nu_r\rangle\cdot \dd W^j_r\rangle\rangle_t\\
&=\langle \langle 2\sum_{j\neq 0} \lambda_j |D| \sum_{i=1}^{3}\int_{0}^{\cdot}  (\langle \sqrt{\nu} e_j z^\nu_r\rangle\times\langle z^\nu_r\rangle)^i\dd W^{j,i}_r \rangle\rangle_t\\
&=4\nu \sum_{j\neq 0} \lambda_j^2 |D|^2 \int_0^t \left|\langle e_j z^\nu_r\rangle\times\langle z^\nu_r\rangle\right|^2\dd r\, .
\end{align*}
Therefore, the evolution of  $g(f_t)$ becomes
\begin{align*}
g(f_t)
&= g(f_0)-2\nu\int_0^t g'(f_r) |D|\langle z^\nu_r|\partial_x z^\nu_r|^2\rangle \cdot \langle z^\nu_r\rangle \dd r\\
&\quad+ 2\nu\sum_{j\neq 0} \lambda_j^2|D| \int_0^t g'(f_r)\left[\langle e_j^2 z^\nu_r \rangle\cdot \langle z^\nu_r\rangle-|\langle  e_j z^\nu_r\rangle |^2 \right]\dd r \\
&\quad-2\sum_{j\neq 0} \lambda_j|D|\int_{0}^{t} g'(f_r) \langle z^\nu_r\rangle\cdot\langle \sqrt{\nu} e_j z^\nu_r\rangle\times \dd W^j_r+2\nu\sum_{j\neq 0} \lambda_j^2|D|^2 \int_0^t g''(f_r) \left|\langle  e_j z^\nu_r\rangle\times\langle z^\nu_r\rangle\right|^2\dd r\, .
\end{align*}
By Theorem \ref{th:local_times_appendix} (iii) applied to the convex map $x\mapsto (x-a)^+$ (where $(x-a)^+:=\sup_{\mathbb{R}}\{(x-a), 0 \}$), by taking the expectation, integrating on $\Gamma$ and from the stationarity of $z^\nu$, the equality holds
\begin{align}\label{eq:proof_eq_1}
\int_{\Gamma} \mathbb{E}\left[\Lambda_t (a)\right]\dd a+ 2 \nu t\int_{\Gamma} \mathbb{E}\left[\mathds{1}_{(a,\infty)}(g(f_t)) \tilde{A}_t\right] \dd a=0\, ,
\end{align}
with 
\begin{align*}
\tilde{A}_t:=  - g'(f_t) |D|\langle z^\nu_t|\partial_x z^\nu_t|^2\rangle \cdot \langle z^\nu_t\rangle& +\sum_{j\neq 0} \lambda_j^2 |D| g'(f_t)\left[\langle e_j^2 z^\nu_t\rangle\cdot \langle z^\nu_t\rangle-|\langle  e_j z^\nu_t\rangle |^2 \right]  \\
&+ \sum_{j\neq 0} \lambda_j^2 |D|^2 g''(f_t) \left|\langle e_j z^\nu_t\rangle\times\langle z^\nu_t\rangle\right|^2\, .
\end{align*}
By Theorem \ref{th:local_times_appendix} (ii), the equality holds
\begin{align}\label{eq:proof_eq_2}
\int_{\Gamma} \mathbb{E}\left[\Lambda_t (a)\right]\dd a = 2\nu t\sum_{j\neq 0} \lambda_j^2|D|^2 \mathbb{E}\left[\mathds{1}_\Gamma (f_t)\left|\langle  e_j z^\nu_t\rangle\times\langle z^\nu_t\rangle\right|^2 g'\left(f_t\right)^2\right]\, .
\end{align}
From \eqref{eq:proof_eq_1} and \eqref{eq:proof_eq_2}, the equality in \eqref{eq:ugly_formula} follows.
\end{proof}
\begin{corollary}\label{cor:null_measure_soace_average}
	For all $t\geq 0$,
	\begin{align*}
	\mathbb{P}(|\langle z^\nu_t \rangle|^2=0 \quad \forall t\geq 0)=0\, , \quad i.e., \quad \mu^\nu(\{u\in H^2: |\langle u \rangle|^2=0\})=0\quad \forall \, \nu\in [0,1)\, .
	\end{align*}
\end{corollary}
\begin{proof}
	Assume that the set $\{\omega\in \Omega: |\langle z^\nu_t(\omega) \rangle|^2=0 \quad \forall t\geq 0 \}$ has positive measure.
	From \eqref{eq:equation_space_average_squared}, for a.e.~$\omega\in \{\omega\in \Omega: |\langle z^\nu_t(\omega) \rangle|^2=0 \quad \forall t\geq 0 \}$, we obtain that
	\begin{align*}
	2\nu \sum_j \lambda_j^2\int_{0}^{t}|\langle  e_j z^\nu_r(\omega)\rangle |^2\dd r=0\, .
	\end{align*} 
	Since $\lambda_j\neq 0$ for all $j\in \mathbb{Z}$, $|\langle  e_j z^\nu_r(\omega)\rangle |^2=0$ for all $j\in \mathbb{Z}$ and $\mathbb{P}$-a.s. From Parseval's identity and since $\mathbb{P}(|\langle z_r\rangle|^2=1)=0$, this contradicts the fact that $\|z^\nu\|^2_{L^2}=|D|$ $\mathbb{P}$-a.s.
\end{proof}
\begin{proposition}\label{pro:ineq} 
Assume that there exists $\lambda_j\neq 0$. Then for all $t\geq 0$
\begin{align*}
 \int_{[\alpha,\beta]}\mathbb{E}\left[\mathds{1}_{(a,\infty)}(\sqrt{f_t}) \frac{L }{2\|z^\nu_t-\langle z^\nu_t\rangle\|_{L^2}}\right] \dd a \leq C(\lambda)[\beta-\alpha]\, ,
\end{align*}
where $f_t:=\|z^\nu_t-\langle z^\nu_t\rangle\|^2_{L^2}$ and $C(\lambda)>0$ is a positive constant depending on the intensity of the noise.
\end{proposition}
\begin{proof}
Fix $\Gamma:=[\alpha,\beta]$, with $0<\alpha< \beta$, and $g(x):=\sqrt{x}$, $g'(x)=1/2\sqrt{x}$ and $g''(x)=-1/4x^{3/2}$ in Proposition \ref{pro:first_equality_Gamma}. Substituting in \eqref{eq:ugly_formula}, for all $t\geq 0$
\begin{align*}
&\quad\int_{[\alpha,\beta]}\mathbb{E}\left[\mathds{1}_{(a,\infty)}(\sqrt{f_t}) \left[\frac{-\langle z^\nu_t|\partial_x z^\nu_t|^2\rangle \cdot \langle z^\nu_t\rangle}{2\|z^\nu_t-\langle z^\nu_t\rangle\|_{L^2}}+\sum_{j\neq 0} \lambda_j^2\frac{\left[\langle e_j^2 z^\nu_t \rangle\cdot \langle z^\nu_t\rangle-|\langle  e_j z^\nu_t\rangle |^2 \right]}{2\|z^\nu_t-\langle z^\nu_t\rangle\|_{L^2}}\right]\right] \dd a\\
&-\int_{[\alpha,\beta]}\sum_{j\neq 0} \lambda_j^2|D|\mathbb{E}\left[\mathds{1}_{(a,\infty)}(\sqrt{f_t}) \frac{\left|\langle e_j z^\nu_t\rangle\times\langle z^\nu_t\rangle\right|^2}{4 \|z^\nu_t-\langle z^\nu_t\rangle\|^3_{L^2}}\right] \dd a\\
&+ \sum_{j\neq 0} \lambda_j^2|D|\mathbb{E}\left[\mathds{1}_{[\alpha, \beta]}(\sqrt{f_t})\frac{\left|\langle  e_j z^\nu_t\rangle\times\langle z^\nu_t\rangle\right|^2}{4 \|z^\nu_t-\langle z^\nu_t\rangle\|^2_{L^2}}\right]=0\, ,
\end{align*}
which leads to the inequality
\begin{equation}\label{eq:inequality_proof}
\begin{aligned}
&\quad\int_{[\alpha,\beta]}\sum_{j\neq 0} \lambda_j^2\mathbb{E}\left[\mathds{1}_{(a,\infty)}(\sqrt{f_t}) \frac{\langle e_j^2 z^\nu_t \rangle\cdot \langle z^\nu_t\rangle }{2\|z^\nu_t-\langle z^\nu_t\rangle\|_{L^2}}\right] \dd a\\
&\leq \int_{[\alpha,\beta]}\mathbb{E}\left[\mathds{1}_{(a,\infty)}(\sqrt{f_t}) \frac{\langle z^\nu_t|\partial_x z^\nu_t|^2\rangle \cdot \langle z^\nu_t\rangle}{2\|z^\nu_t-\langle z^\nu_t\rangle\|_{L^2}}\right] \dd a+\int_{[\alpha,\beta]}\sum_{j\neq 0} \lambda_j^2|D|\mathbb{E}\left[\mathds{1}_{(a,\infty)}(\sqrt{f_t}) \frac{\left|\langle  e_j z^\nu_t\rangle\times\langle z^\nu_t\rangle\right|^2}{4 \|z^\nu_t-\langle z^\nu_t\rangle\|^3_{L^2}}\right] \dd a\\
&\quad+ \int_{[\alpha,\beta]}\sum_{j\neq 0} \lambda_j^2\mathbb{E}\left[\mathds{1}_{(a,\infty)}(\sqrt{f_t}) \frac{|\langle  e_j z^\nu_t\rangle |^2 }{2\|z^\nu_t-\langle z^\nu_t\rangle\|_{L^2}}\right] \dd a\, .
\end{aligned}
\end{equation}
For all $t\geq 0$, the inequality holds
\begin{align}\label{eq:ineq_space_av_proof_1}
\|\partial_x z^\nu_t\|^2_{L^2}= \|\partial_x [z^\nu_t-\langle z^\nu_t \rangle]\|^2_{L^2}=-(z^\nu_t-\langle z^\nu_t \rangle, \partial^2_x z^\nu_t)_{L^2}\leq \|z^\nu_t-\langle z^\nu_t \rangle\|_{L^2}\|\partial^2_x z^\nu_t\|_{L^2}\, .
\end{align}
We estimate the right-hand side of \eqref{eq:inequality_proof} from above and the first integral is uniformly bounded in $\nu$ 
\begin{align*}
\quad\int_{[\alpha,\beta]}\mathbb{E}\left[\mathds{1}_{(a,\infty)}(f_t) \frac{\langle z^\nu_t|\partial_x z^\nu_t|^2\rangle \cdot \langle z^\nu_t\rangle}{2\|z^\nu_t-\langle z^\nu_t\rangle\|_{L^2}}\right] \dd a
&\leq \mathbb{E}\left[\frac{\|\partial_x z^\nu_t\|^2_{L^2}}{2|D| \|z^\nu_t-\langle z^\nu_t\rangle\|_{L^2}}\right] \\
&\leq \mathbb{E}\left[\frac{\|\partial^2_x z^\nu_t\|_{L^2}\|z^\nu_t-\langle z^\nu_t\rangle\|_{L^2}}{2|D| \|z^\nu_t-\langle z^\nu_t\rangle\|_{L^2}}\right]\\
&\leq \frac{\mathbb{E}\left[\|\partial^2_x z^\nu_t\|_{L^2}\right]}{2|D|}[\beta-\alpha]\, ,
\end{align*}
where the uniform bound follows from  \eqref{eq:ineq_space_av_proof_1} and Lemma \ref{lemma:unif_bounds_nu}. We estimate the second integral on the right-hand side of \eqref{eq:inequality_proof}: observe that from $a\times b\cdot a=0$ for all $a,b\in \mathbb{R}^3$,
\begin{align*}
\langle e_j z^\nu_t\rangle\times\langle z^\nu_t\rangle=
\langle e_j [z^\nu_t-\langle z^\nu_t\rangle]\rangle\times\langle z^\nu_t\rangle+\langle e_j \langle z^\nu_t\rangle \rangle\times\langle z^\nu_t\rangle=\langle e_j [z^\nu_t-\langle z^\nu_t\rangle]\rangle\times\langle z^\nu_t\rangle\, ,
\end{align*}
from which the inequality holds 
\begin{align*}
\left|\langle e_j [z^\nu_t-\langle z^\nu_t\rangle]\rangle\times\langle z^\nu_t\rangle\right|^2\leq \frac{\|e_j\|^2_{L^2}\|z^\nu_t-\langle z^\nu_t\rangle\|^2_{L^2}}{|D|^2}=\frac{\langle e_j^2\rangle \|z^\nu_t-\langle z^\nu_t\rangle\|^2_{L^2}}{|D|}
\end{align*}
and  the estimate becomes
\begin{align*}
&\quad\sum_{j\neq 0} \lambda_j^2|D|\int_{[\alpha,\beta]}\mathbb{E}\left[\mathds{1}_{(a,\infty)}(f_t) \frac{\left|\langle e_j [z^\nu_t-\langle z^\nu_t\rangle]\rangle\times\langle z^\nu_t\rangle\right|^2}{4 \|z^\nu_t-\langle z^\nu_t\rangle\|^3_{L^2}}\right] \dd a\\
&\leq \sum_{j\neq 0} \lambda_j^2\int_{[\alpha,\beta]}\mathbb{E}\left[\mathds{1}_{(a,\infty)}(f_t) \frac{|D|\langle e_j^2 \rangle}{4|D| \|z^\nu_t-\langle z^\nu_t\rangle\|_{L^2}}\right] \dd a\, ,
\end{align*}
which is not uniformly bounded in $\nu$ and we will need to absorb it into the left-hand side of \eqref{eq:inequality_proof}. We estimate the third integral on the right-hand side of \eqref{eq:inequality_proof} as follows
\begin{align*}
&\quad\sum_{j\neq 0} \lambda_j^2\int_{[\alpha,\beta]}\mathbb{E}\left[\mathds{1}_{(a,\infty)}(f_t) \frac{|\langle  e_j z^\nu_t\rangle |^2 }{2\|z^\nu_t-\langle z^\nu_t\rangle\|_{L^2}}\right] \dd a \\
&=\sum_{j\neq 0} \lambda_j^2 \int_{[\alpha,\beta]}\mathbb{E}\left[\mathds{1}_{(a,\infty)}(f_t) \left[\frac{\langle  e_j (z^\nu_t-\langle z^\nu_t\rangle)\rangle\cdot \langle  e_j z^\nu_r\rangle  }{2\|z^\nu_t-\langle z^\nu_t\rangle\|_{L^2}}+\frac{\langle  e_j \rangle \langle z^\nu_t\rangle\cdot \langle  e_j (z^\nu_t-\langle z^\nu_t\rangle)\rangle  }{2\|z^\nu_t-\langle z^\nu_t\rangle\|_{L^2}}\right]\right] \dd a \\
&\quad + \sum_{j\neq 0} \lambda_j^2 \int_{[\alpha,\beta]}\mathbb{E}\left[\mathds{1}_{(a,\infty)}(f_t) \frac{|\langle  e_j \rangle|^2 |\langle z^\nu_t\rangle|^2 }{2\|z^\nu_t-\langle z^\nu_t\rangle\|_{L^2}}\right]\dd a \\
&\leq C(\lambda)[\beta-\alpha]+\sum_{j\neq 0} \lambda_j^2\int_{[\alpha,\beta]}\mathbb{E}\left[\mathds{1}_{(a,\infty)}(f_t) \frac{|\langle  e_j \rangle|^2  }{2\|z^\nu_t-\langle z^\nu_t\rangle\|_{L^2}}\right]\dd a\leq C(\lambda)[\beta-\alpha]\, ,
\end{align*}
where, in the last inequality we use that, from the construction of the orthonormal basis, $\langle e_j\rangle =0$  for all $j\in \mathbb{Z}\setminus \{0\}$.
The constant $C(\lambda)>0$ depends on the intensity of the noise and may change from line to line.

We rewrite the integral on the left-hand side of \eqref{eq:inequality_proof}. Notice that for all fixed $j\in \mathbb{N}$
\begin{align*}
\langle e_j^2 z^\nu_t \rangle\cdot \langle z^\nu_t\rangle& =\frac{1}{|D|}\int_D z^\nu_t \cdot e_j^2\langle z^\nu_t\rangle \dd x\\
&=\frac{1}{|D|}\int_D \left[ z^\nu_t-\langle z^\nu_t  \rangle \right]\cdot e_j^2\langle z^\nu_t\rangle  \dd x+\frac{1}{|D|}\int_D \langle z^\nu_t \rangle\cdot e_j^2\langle z^\nu_t\rangle\dd x\\
&=\frac{1}{|D|}\int_D \left[ z^\nu_t -\langle z^\nu_t  \rangle \right]\cdot e_j^2\langle z^\nu_t\rangle  \dd x+ |\langle z^\nu_t\rangle|^2\langle e_j^2\rangle \\
&=\frac{1}{|D|}\int_D \left[ z^\nu_t -\langle z^\nu_t  \rangle \right]\cdot e_j^2\langle z^\nu_t\rangle  \dd x +[ |\langle z^\nu_t \rangle|^2-1]\langle e_j^2\rangle+\langle e_j^2\rangle \\
&=\langle \left[ z^\nu_t -\langle z^\nu_t \rangle \right]\cdot e_j^2\langle z^\nu_t\rangle  \rangle -\frac{\|z^\nu_t -\langle z^\nu_t \rangle \|^2_{L^2}}{|D|}+\langle e_j^2\rangle\, .
\end{align*}
Observe that the following integrals are uniformly bounded in $\nu$
\begin{align*}
\sum_{j\neq 0} \lambda_j^2\int_{[\alpha,\beta]}\mathbb{E}\left[\mathds{1}_{(a,\infty)}(f_t) \frac{|\langle \left[ z^\nu_t -\langle z^\nu_t \rangle \right]\cdot e_j^2\langle z^\nu_t\rangle  \rangle| +|D|^{-1}\|z^\nu_t -\langle z^\nu_t\rangle \|^2_{L^2}}{2\|z^\nu_t-\langle z^\nu_t\rangle\|_{L^2}}\right] \dd a\leq C(\lambda)[\beta-\alpha]\, ,
\end{align*}
thus, we move them to the right-hand side of \eqref{eq:inequality_proof}. Summarizing the above considerations, we rewrite \eqref{eq:inequality_proof} as 
\begin{align}
&\quad\quad\sum_{j\neq 0} \lambda_j^2\int_{[\alpha,\beta]}\mathbb{E}\left[\mathds{1}_{(a,\infty)}(f_t) \frac{\langle e_j^2\rangle}{2\|z^\nu_t-\langle z^\nu_t\rangle\|_{L^2}}\right] \dd a  \label{eq:proof_space_av_3}\\
&\leq C(\lambda)[\beta-\alpha]
+ \sum_{j\neq 0} \lambda_j^2\int_{[\alpha,\beta]}\mathbb{E}\left[\mathds{1}_{(a,\infty)}(f_t) \frac{\langle  e_j^2 \rangle}{4 \|z^\nu_t-\langle z^\nu_t\rangle\|_{L^2}}\right] \dd a\, . \label{eq:proof_space_av_2}
\end{align}
By absorbing the term \eqref{eq:proof_space_av_2} to the left-hand side of the inequality \eqref{eq:proof_space_av_3}, it holds
\begin{align*}
\sum_{j\neq 0} \lambda_j^2\int_{[\alpha,\beta]}\mathbb{E}\left[\mathds{1}_{(a,\infty)}(f_t) \frac{\langle e_j^2\rangle }{4\|z^\nu_t-\langle z^\nu_t\rangle\|_{L^2}}\right] \dd a  \leq C(\lambda)[\beta-\alpha]\, . 
\end{align*}
Since $\lambda_j\neq 0$ for all $j\in \mathbb{N}$, the lower bound is positive and the assertion follows.
\end{proof}
With an analogous proof to that of \ref{pro:ineq}, we establish the following result.
\begin{theorem}\label{th:small_pr_small_B}
\textbf{Small probability of small balls.} 
Assume that there exists $\lambda_k\neq 0$. Then for all $t\geq 0$, 
\begin{itemize}
\item[(i)]  for all fixed $\nu\in [0,1)$
\begin{align*}
\mu_\nu(\{u\in H^1:\|u-\langle u \rangle\|_{L^2}\leq \delta\})\leq \frac{8C(\lambda)}{\lambda_k^2 \langle e_k^2\rangle} \delta\, , \quad \forall \delta >0\, ,
\end{align*}
\item[(ii)] for all fixed $\nu\in [0,1)$ and for all $\delta >0$
\begin{align*}
\mu_\nu(\{u\in H^1:\|\nabla u\|^2_{L^2}\leq \delta\})\leq C \delta\, .
\end{align*}
\end{itemize}

\end{theorem}
\begin{proof}
\textit{Proof of (i).}
By passing to the limit for $\alpha \rightarrow 0^+$ and using the assumption that there exists $\lambda_{k}>0$, the inequality in Proposition \ref{pro:ineq} reads
\begin{align*}
&\quad \int_{[0,\beta]}\mathbb{E}\left[\mathds{1}_{(a,\infty)}(\|z^\nu_t-\langle z^\nu_t\rangle\|_{L^2}) \frac{L }{4 \|z^\nu_t-\langle z^\nu_t\rangle\|_{L^2}}\right] \dd a \leq 2 C(\lambda)\beta\, .
\end{align*}
We fix $\delta>0$ and we lower-bound the integral by
\begin{align*}
&\quad\int_{[0,\beta]}\mathbb{E}\left[\mathds{1}_{(a,\infty)}(\|z^\nu_t-\langle z^\nu_t\rangle\|_{L^2}) \frac{1 }{\|z^\nu_t-\langle z^\nu_t\rangle\|_{L^2}}\right] \dd a\\
&\geq \int_{[0,\beta]}\mathbb{E}\left[\mathds{1}_{(a,\delta]}(\|z^\nu_t-\langle z^\nu_t\rangle\|_{L^2}) \frac{1 }{\|z^\nu_t-\langle z^\nu_t\rangle\|_{L^2}}\right] \dd a\\
&\geq \frac{1}{\delta } \int_{[0,\beta]}\mathbb{E}\left[\mathds{1}_{(a,\delta]}(\|z^\nu_t-\langle z^\nu_t\rangle\|_{L^2}) \right] \dd a\, .
\end{align*}
Thus, the inequality holds:
\begin{align*}
\frac{1}{\beta } \int_{[0,\beta]}\mathbb{P}(a< \|z^\nu_t-\langle z^\nu_t\rangle\|_{L^2} \leq \delta) \dd a \leq 2 C(\lambda)\delta\, .
\end{align*}
We pass to the limit as $a \rightarrow 0^+$ and use the fact that the measure $\mu^\nu$ has no atom at $0$, from Lemma \ref{lemma:unif_bounds_nu}. The claim \textit{(i)} follows.\\

\textit{Proof of (ii).}
By the Poincaré-Wirtinger inequality, it holds for all $t\geq 0$ that
\begin{align*}
\|z^\nu_t-\langle z^\nu_t\rangle\|^2_{L^2} \leq C_P \|\partial_x z^\nu_t\|^2_{L^2} \leq C_p^2  \|\partial^2_x z^\nu_t\|^2_{L^2}\, , \quad\quad \mathbb{P}-\mathrm{a.s.}
\end{align*}
For all $\delta>0$, the following inclusion holds 
\begin{align*}
\{\omega\in \Omega: C_P \|\partial^2_x z^\nu_t\|^2_{L^2}\leq \delta \} \subset\{\omega\in \Omega: C_P \|\partial_x z^\nu_t\|^2_{L^2}\leq \delta \} \subset \{\omega\in \Omega: \|z^\nu_t-\langle z^\nu_t\rangle\|_{L^2} \leq \delta \} \, .
\end{align*}
\end{proof}

\begin{proposition}\label{pro:inequ_grad_leb} For every Borel subset $A\subset \mathbb{R}$, it holds 
\begin{align*}
\mathbb{P}(\|\partial_x z^\nu_t\|_{L^2}^2\in A )\lesssim |A|^{1/2}\, , \quad \forall t\geq 0\, ,
\end{align*}
where the constant is independent of $\nu\in (0,1]$.
\end{proposition}
\begin{proof}
For $\epsilon>0$, we define the set 
\begin{align*}
	\Omega^c_\epsilon:=\{\omega\in \Omega:\|\partial_x z^\nu_t(\omega)\|_{L^2}^2>\epsilon^{1/2}, \quad  \|z^\nu_t(\omega)\times \partial_x^2 z^\nu_t(\omega)\|^2_{L^2}\leq \epsilon^{-1/2}\}\, .
\end{align*}
From Theorem \ref{th:small_pr_small_B} and from the Markov inequality, 
\begin{equation}\label{eq:good_set_inequalities}
\begin{aligned}
\mathbb{P}(\{ \|\partial_x z^\nu_t(\omega)\|_{L^2}^2\in A\} \cap \Omega_\epsilon) &\leq \mathbb{P}(\|\partial_x z^\nu_t\|_{L^2}^2\in  \Omega_\epsilon)\\
&=\mathbb{P}(\|\partial_x z^\nu_t\|_{L^2}^2\leq \sqrt{\epsilon})+\mathbb{P}(\|z^\nu_t \times \partial_x^2 z^\nu_t\|^2_{L^2}> \epsilon^{-1/2})\\
&\leq C \sqrt{\epsilon} +\sqrt{\epsilon} \mathbb{E}[\|z^\nu_t \times \partial_x^2 z^\nu_t\|^2_{L^2}]\leq (C+L^2) \sqrt{\epsilon}\, .
\end{aligned}	
\end{equation}
Assume now that there exists $k(\epsilon)=c\sqrt{\epsilon}$ such that
\begin{align}\label{eq:ansatz}
\mathbb{P}(\{ \|\partial_x z^\nu_t(\omega)\|_{L^2}^2\in A\}\cap \Omega^c_\epsilon)\leq k(\epsilon)^{-1} |A|\, .
\end{align}
From \eqref{eq:good_set_inequalities} and \eqref{eq:ansatz}, it follows
\begin{align*}
\mathbb{P}(\|\partial_x z^\nu_t(\omega)\|_{L^2}^2\in A )&=\mathbb{P}(\{ \|\partial_x z^\nu_t(\omega)\|_{L^2}^2\in A\} \cap \Omega_\epsilon)+\mathbb{P}(\{ \|\partial_x z^\nu_t(\omega)\|_{L^2}^2\in A\}\cap \Omega^c_\epsilon)\lesssim p(|A|)\, ,
\end{align*}
where $p$ is a continuous monotone non-decreasing function such that $p(0)=0$.

The remaining part of the proof is devoted to showing inequality \eqref{eq:ansatz}.
The proof follows the arguments in \cite[Proposition 3.3]{Shirikyan_local_times}: many differences emerge due to the structure of the nonlinearity and the noise. We fix $g\in C^2(\mathbb{R};\mathbb{R})$ and we describe the evolution of $g(\|\partial_x z^\nu_t\|^2_{L^2})$ by It\^o's formula as, 
\begin{align*}
g(\|\partial_x &z^\nu_t\|^2_{L^2})= g(\|\partial_x z^\nu_0\|^2_{L^2})-2\nu\int_0^t g'(\|\partial_x z^\nu_r\|^2_{L^2}) \|z^\nu_r\times \partial_x^2 z^\nu_r\|^2_{L^2} \dd r\\
&+ 2\nu \sum_{j\neq 0} \lambda_j^2\int_0^t g'(\|\partial_x z^\nu_r\|^2_{L^2})\|\partial_x e_j\|^2_{L^2}\dd r \\
&+2\sum_{j\neq 0} \lambda_j\int_{0}^{t} g'(\|\partial_x z^\nu_r\|^2_{L^2}) \left(\partial_x z^\nu_r,\sqrt{\nu} \partial_x e_j z^\nu_r\times \dd W^j_r\right)+\frac{ 1}{2} \int_0^t g''(\|\partial_x z^\nu_r\|^2_{L^2}) \dd\langle \langle \|\partial_x z^\nu_{\cdot}\|^2_{L^2}\rangle\rangle_r\, .
\end{align*}
The quadratic variation $\langle \langle \|\partial_x z^\nu_t\|^2_{L^2} \rangle\rangle_t$ reads
\begin{align*}
\dd \langle \langle \|\partial_x z^\nu_\cdot\|^2_{L^2} \rangle \rangle_{t} &=  \dd \langle \langle 2 \sum_{j\neq 0} \lambda_j\int_0^{\cdot} \int_{D}\sqrt{\nu} \partial_x e_j z^\nu_r\times \partial_x z^\nu_r\cdot  \dd W^j_r \rangle \rangle_{t}\\
&=4\nu \sum_{j\neq 0} \lambda_j^2 \left|\int_{D} \partial_x e_j z^\nu_r\times \partial_x z^\nu_r\dd x\right|^2\dd r=:2\nu Q_r\dd r\, .
\end{align*}
For convenience, denote for all $r\geq 0$
\begin{align*}
	\mathcal{D}_r:=-2\nu g'(\|\partial_x z^\nu_r\|^2_{L^2}) \|z^\nu_r\times \partial_x^2 z^\nu_r\|^2_{L^2} + 2\nu \sum_{j\neq 0} \lambda_j^2 g'(\|\partial_x z^\nu_r\|^2_{L^2})\|\partial_x e_j\|^2_{L^2}+2 \nu g''(\|\partial_x z^\nu_r\|^2_{L^2}) Q_r\, .
\end{align*}
From the elementary identity $a\cdot b\times c=c\cdot a\times b\,$ for all $a,b,c\in \mathbb{R}^3$, we rewrite the stochastic integral as
\begin{align*}
\int_{0}^{t} g'(\|\partial_x z^\nu_r\|^2_{L^2}) \left(\partial_x z^\nu_r,\sqrt{\nu} \partial_x e_j z^\nu_r\times \dd W^j_r\right)&=\int_{0}^{t} g'(\|\partial_x z^\nu_r\|^2_{L^2}) \left(\sqrt{\nu} \partial_x e_j\partial_x z^\nu_r\times z^\nu_r, \dd W^j_r\right)\\
&=\int_{0}^{t}\sum_{i=1}^3 g'(\|\partial_x z^\nu_r\|^2_{L^2}) \left(\sqrt{\nu} \partial_xe_j (\partial_x z^\nu_r\times z^\nu_r)^i \dd W^{j,i}_r\right)\, ,
\end{align*}
for every $j\in \mathbb{Z}\setminus \{0\}$.
Denote by $\Lambda_t(a)$ the local time associated to $g(\|\partial_x z^\nu_t\|^2_{L^2})$ and by applying Theorem \ref{th:local_times_appendix} with $h=1_{\Gamma}$ and $\Gamma:=A$, 
\begin{align*}
\int_{A} \mathbb{E}\left[\Lambda_t(a)\right]\dd a 
=2 \nu t \sum_{j\neq 0} \lambda_j^2\mathbb{E}\left[1_A(\|\partial_x z^\nu_t\|^2_{L^2})|g'(\|\partial_x z^\nu_t\|^2_{L^2}) |^2 \left|\int_D \partial_x e_j z^\nu_t\times \partial_x z^\nu_t\dd x\right|^2\right]\, .
\end{align*}
From the stationarity of $z^\nu$ on $H^1(\mathbb{S}^2)$ and the change of variable in Theorem \ref{th:local_times_appendix}, $\mathbb{E}[\Lambda_t(a)]=-\nu t\mathbb{E}[1_{(a,\infty)}(\|\partial_x z^\nu_t\|^2_{L^2})\mathcal{D}_t]$. Putting together these two relations, we conclude that for all $t\geq 0$
\begin{equation}\label{eq:equality_local_times_a}
\begin{aligned}
&\int_A \mathbb{E}[1_{(a,\infty)} (g(\|\partial_x z^\nu_t\|^2_{L^2})) 2\nu( \sum_{j\neq 0}\lambda_j^2\|\partial_x e_j\|^2_{L^2}-\|z^\nu_t\times \partial_x^2 z^\nu_t\|^2_{L^2} )g'(\|\partial_x z_t^\nu\|^2_{L^2} )
\\
&\quad+2\sum_{j\neq 0}\lambda_j^2 g''(\|\partial_x z^\nu_t\|^2_{L^2} )\nu \left|\int_{D} \partial_x e_j z^\nu_t\times \partial_x z^\nu_t\dd x\right|^2]\dd a
\\
&\quad+2\nu\sum_{j\neq 0}\lambda_j^2 \mathbb{E}\left[1_A(\|\partial_x z^\nu_t\|^2_{L^2})|g'(\|\partial_x z^\nu_t\|^2_{L^2}) |^2 \left|\int_D \partial_x e_j z^\nu_t\times \partial_x z^\nu_t\dd x\right|^2\right]=0\, .
\end{aligned}
\end{equation}
We substitute $g(x)=x$ with $g'(x)=1$ and $g''(x)=0$ into \eqref{eq:equality_local_times_a}:
\begin{equation}\label{eq:equality_local_times_a_2}
\begin{aligned}
\int_A &\mathbb{E}[1_{(a,\infty)} (\|\partial_x z^\nu_t\|^2_{L^2})( \sum_{j\neq 0}\lambda_j^2\|\partial_x e_j\|^2_{L^2}-\|z^\nu_t\times \partial_x^2 z^\nu_t\|^2_{L^2} )]\dd a\\
&\quad+\sum_{j\neq 0}\lambda_j^2\mathbb{E}\left[1_A(\|\partial_x z^\nu_t\|^2_{L^2}) \left|\int_D \partial_xe_j z^\nu_t\times \partial_x z^\nu_t\dd x\right|^2\right]=0\, .
\end{aligned}
\end{equation}
From \eqref{eq:equality_local_times_a_2}, the inequality holds for all $t\geq 0$:
\begin{align}\label{eq:left_HS}
&\quad\sum_{j\neq 0}\lambda_j^2\mathbb{E}\left[1_A(\|\partial_x z^\nu_t\|^2_{L^2}) \left|\int_D \partial_x e_j z^\nu_t\times \partial_x z^\nu_t\dd x\right|^2\right]\\
&\leq \int_A \mathbb{E}[1_{(a,\infty)} (\|\partial_x z^\nu\|^2_{L^2}) \|z^\nu_t\times \partial_x^2 z^\nu_t\|^2_{L^2} ]\dd a
\leq L\lambda(A) \nonumber\, ,
\end{align}
where the upper bound follows from Lemma \ref{lemma:unif_bounds_nu}.
In the introduction, we have considered the orthonormal basis of $L^2(D)$ given by \eqref{eq:ONB}, where in particular $e_0=1$. 
Observe now that the sequence $\{e_k\}_{\mathbb{Z}\setminus 0}$ is an orthonormal basis of $L^2_0(D;\mathbb{R})$ from the Strum-Liouville theorem and notice that $u\times \partial_x u\in L^2_0(D;\mathbb{R}^3)$. 
We now bound \eqref{eq:left_HS} from below. We notice that 
\begin{align}
&\sum_{j\neq 0}\lambda_j^2\mathbb{E}\left[1_A(\|\partial_x z^\nu_t\|^2_{L^2}) \left|\int_D \partial_x e_j z^\nu_t\times \partial_x z^\nu_t\dd x\right|^2\right] \nonumber\\
&\geq \sum_{j\neq 0}\lambda_j^2\mathbb{E}\left[1_{\Omega^c_\epsilon}1_A(\|\partial_x z^\nu_t\|^2_{L^2}) \left|\int_D \partial_x e_j z^\nu_t\times \partial_x z^\nu_t\dd x\right|^2\right] \nonumber\\
&\geq \lambda_N^2 \sum_{0<|j|\leq N} \mathbb{E}\left[1_{\Omega^c_\epsilon} 1_A(\|\partial_x z^\nu_t\|^2_{L^2}) \left|\int_D \partial_x e_j z^\nu_t\times \partial_x z^\nu_t\dd x\right|^2\right]\nonumber\\
&=\lambda_N^2 \sum_{0<|j|\leq N} j^2 \mathbb{E}\left[1_{\Omega^c_\epsilon}1_A(\|\partial_x z^\nu_t\|^2_{L^2}) \left|\int_D  e_j z^\nu_t\times \partial_x z^\nu_t\dd x\right|^2\right] \nonumber\\
&\geq \lambda_N^2 \sum_{0<|j|\leq N}  \mathbb{E}\left[1_{\Omega^c_\epsilon}1_A(\|\partial_x z^\nu_t\|^2_{L^2}) \left|\int_D  e_j z^\nu_t\times \partial_x z^\nu_t\dd x\right|^2\right] \nonumber\\
&= \lambda_N^2 \mathbb{E}\left[1_{\Omega^c_\epsilon}1_A(\|\partial_x z^\nu_t\|^2_{L^2}) \| z^\nu_t\times \partial_x z^\nu_t\|_{L^2}^2\right]-\lambda_N^2 \sum_{|j|>N} \mathbb{E}\left[1_{\Omega^c_\epsilon}1_A(\|\partial_x z^\nu_t\|^2_{L^2}) \left|\int_D  e_j z^\nu_t\times \partial_x z^\nu_t\dd x\right|^2\right] \nonumber\\
&\geq \lambda_N^2 \mathbb{E}\left[1_{\Omega^c_\epsilon}1_A(\|\partial_x z^\nu_t\|^2_{L^2}) \| z^\nu_t\times \partial_x z^\nu_t\|_{L^2}^2\right]-\frac{\lambda_N^2}{N^2}  \mathbb{E}\left[1_{\Omega^c_\epsilon}1_A(\|\partial_x z^\nu_t\|^2_{L^2}) \| z^\nu_t\times \partial^2_x z^\nu_t\|_{L^2}^2\right]\, ,\label{eq:lower_bound_Gamma}
\end{align}
where we use that $\| z^\nu_t\times \partial_x z^\nu_t\|_{L^2}^2=\| \partial_x z^\nu_t\|_{L^2}^2$.

Since we are on $\Omega^c_\epsilon$, we can lower bound \eqref{eq:lower_bound_Gamma} by
\begin{align*}
\lambda_N^2 (\epsilon^{1/2} - \frac{\epsilon^{-1/2}}{N^2})\mathbb{P}(\{\|\partial_x z^\nu_t\|_{L^2}^2\in A \}\cap \Omega^c_\epsilon)\, .
\end{align*}
We now choose $N\equiv N(\epsilon)>0$ such that
\begin{align*}
 N^{-2}\epsilon^{-1/2}= \frac{\epsilon^{1/2}}{2 } \implies N^2=\frac{2}{\epsilon}\, .
\end{align*}
Thus
\begin{align*}
C_p\lambda_N^2 (\epsilon^{1/2} -N^{-2} \epsilon^{-1/2})= C_p\lambda_N^2 \frac{\sqrt{\epsilon}}{2}=:k(\epsilon)\, ,
\end{align*}
and $k(\epsilon)$ converges to $0$ for $\epsilon\rightarrow +\infty$. In conclusion
\begin{align*}
\mathbb{P}(\{\|\partial_x z^\nu_t\|_{L^2}^2\in A\} \cap \Omega^c_\epsilon)\leq k(\epsilon)^{-1} |A|\, 
\end{align*}
and the assertion follows.

\end{proof}
\begin{proposition}\label{pro:space_av_small}For all $\delta>0$, it holds
\begin{align}\label{eq:small_space_average}
\mathbb{P}(|\langle z^\nu_t \rangle|^2<\delta)<\delta\, ,\quad \forall t\geq 0\, ,
\end{align}
where the constant is independent of $\nu\in [0,1)$.
\end{proposition}
\begin{proof}
	\textit{ Proof of the inequality if $f_t=|\langle z^\nu_t \rangle|^2$.}
	For all $\gamma>0$, define the set
	\begin{align*}
	\Omega_{\gamma}:=\{\omega\in \Omega: \|z^\nu_t-\langle  z^\nu_t\rangle\|_{L^2}^2\geq\gamma\, , \quad \|Q_N \partial_x z^\nu_t\|^2_{L^2}\leq 1/\gamma  \}\, .
	\end{align*}
	We aim to show that for all $\delta>0$, the inequality \eqref{eq:small_space_average} holds. To this aim, we decompose 
	\begin{align*}
	\mathbb{P}(|\langle z^\nu_t\rangle|^2<\delta)=\mathbb{P}(\{|\langle z^\nu_t \rangle|^2<\delta\}\cap \Omega^\gamma)+\mathbb{P}(\{|\langle z^\nu_t \rangle|^2<\delta\}\cap \Omega^{\gamma,C})\, .
	\end{align*}
	From analogous computations as in \eqref{eq:good_set_inequalities}, $\mathbb{P}(\{|\langle z^\nu_t \rangle|^2<\delta\}\cap \Omega^{\gamma,C})<\gamma$.
	The bound of $\mathbb{P}(\{|\langle z^\nu_t \rangle|^2<\delta\}\cap \Omega^\gamma)$ is more involved. The proof is similar to the proof in Proposition \ref{pro:ineq}.
	We apply Proposition \ref{pro:first_equality_Gamma} and the equality for $g(x)=\sqrt{x}$ and $f_t=|\langle z^\nu_t \rangle|^2$ reads, for all $t\geq 0$
	\begin{align*}
	&\quad\int_{[\alpha,\beta]}\mathbb{E}\left[\mathds{1}_{(a,\infty)}(\sqrt{f_t}) \left[\frac{\langle z^\nu_t|\partial_x z^\nu_t|^2\rangle \cdot \langle z^\nu_t\rangle}{2|\langle z^\nu_t\rangle|}-\sum_{j\neq 0} \lambda_j^2\frac{\left[\langle e_j^2 z^\nu_t \rangle\cdot \langle z^\nu_t\rangle-|\langle  e_j z^\nu_t\rangle |^2 \right]}{2|\langle z^\nu_t \rangle|}\right]\right] \dd a\\
	&+\int_{[\alpha,\beta]}\sum_{j\neq 0} \lambda_j^2|D|\mathbb{E}\left[\mathds{1}_{(a,\infty)}(\sqrt{f_t}) \frac{\left|\langle e_j z^\nu_t\rangle\times\langle z^\nu_t\rangle\right|^2}{4 |\langle z^\nu_t \rangle|^3}\right] \dd a\\
	&+ \sum_{j\neq 0} \lambda_j^2|D|\mathbb{E}\left[\mathds{1}_{[\alpha,\beta]}(\sqrt{f_t})\frac{\left|\langle  e_j z^\nu_t\rangle\times\langle z^\nu_t\rangle\right|^2}{4 |\langle z^\nu_t \rangle|^2}\right]=0\, ,
	\end{align*}
	which leads to
	\begin{align*}
	&\quad\int_{[\alpha,\beta]}\mathbb{E}\left[\mathds{1}_{(a,\infty)}(\sqrt{f_t}) \left[\frac{\langle z^\nu_t|\partial_x z^\nu_t|^2\rangle \cdot \langle z^\nu_t\rangle}{2|\langle z^\nu_t \rangle|}-\sum_{j\neq 0} \lambda_j^2\frac{\left[\langle e_j^2 z^\nu_t \rangle\cdot \langle z^\nu_t\rangle-|\langle  e_j z^\nu_t\rangle |^2 \right]}{2|\langle z^\nu_t \rangle|}\right]\right] \dd a\leq 0\, ,
	\end{align*}
	which we rewrite as
	\begin{align*}
    &\quad\int_{[\alpha,\beta]}\mathbb{E}\left[\mathds{1}_{(a,\infty)}(\sqrt{f_t}) \sum_{j\neq 0} \lambda_j^2\frac{|\langle  e_j z^\nu_t\rangle |^2 }{2|\langle z^\nu_t \rangle|}\right]\dd a\\
	&\leq \int_{[\alpha,\beta]}\mathbb{E}\left[\mathds{1}_{(a,\infty)}(\sqrt{f_t}) \left[\frac{-\langle z^\nu_t|\partial_x z^\nu_t|^2\rangle \cdot \langle z^\nu_t\rangle}{2|\langle z^\nu_t \rangle|}+\sum_{j\neq 0} \lambda_j^2\frac{\langle e_j^2 z^\nu_t \rangle\cdot \langle z^\nu_t\rangle}{2|\langle z^\nu_t \rangle|}\right]\right] \dd a\leq C \sum_{j \neq 0}\lambda_j^2 (\beta-\alpha)
	\, .
	\end{align*}
	Remark that 
	\begin{align*}
	\sum_{j\neq 0}\lambda_j^2 |\langle  e_j z^\nu_t\rangle |^2 \geq 2\pi\lambda_N [1-|\langle  z^\nu_t\rangle |^2-\|Q_N z^\nu_t\|^2_{L^2}]= 2 \pi\lambda_N [\|z^\nu_t-\langle  z^\nu_t\rangle\|_{L^2}^2-\|Q_N z^\nu_t\|^2_{L^2}]\, .
	\end{align*}
	On the set $\Omega_{\gamma}:=\{\omega\in \Omega: \|z^\nu_t-\langle  z^\nu_t\rangle\|_{L^2}^2>\sqrt{\gamma}\, , \quad \|Q_N \partial_x z^\nu_t\|^2_{L^2}<1/\sqrt{\gamma}  \}$, it holds
	\begin{align*}
	\sum_{j\neq 0}\lambda_j^2 |\langle  e_j z^\nu_t\rangle |^2 &\geq 2\pi\lambda_N^2 [\|z^\nu_t-\langle  z^\nu_t\rangle\|_{L^2}^2-\frac{1}{N^2}\|Q_N \partial_x z^\nu_t\|^2_{L^2}]\geq 2\pi \lambda_N^2 \left(\sqrt{\gamma} -\frac{1}{N^2\sqrt{\gamma}}\right)\, ,
	\end{align*}
	where, by choosing $N^2=2/\gamma$, the lower bound is $k(\gamma):=\lambda_N^2\sqrt{\gamma}/2$. Thus, for all $\gamma>0$, 
	\begin{align*}
	 \frac{1}{\delta } \int_{[\alpha,\beta]}\mathbb{E}\left[\mathds{1}_{\Omega_{\gamma}} \mathds{1}_{(a,\delta]}(|\langle z^\nu_t \rangle|) \right] \dd a&\leq \int_{[\alpha,\beta]}\mathbb{E}\left[\mathds{1}_{\Omega_{\gamma}}\mathds{1}_{(a,\infty)}(|\langle z^\nu_t \rangle|) \frac{1 }{|\langle z^\nu_t \rangle|^2}\right]\dd a\\
	 &\leq \frac{2 C}{2 \pi\sum_{j\neq 0} \lambda_j^2 k(\gamma)} \sum_{j \neq 0}\lambda_j^2 (\beta-\alpha)
	\, .
	\end{align*}
	where the left-hand side bound holds from considerations completely analogous to Theorem \eqref{th:small_pr_small_B} (i). By passing to the limit in $\alpha\rightarrow 0^+$, multiplying by $\delta/\beta$ and finally by passing to the limit in $\beta\rightarrow 0^+$, the bound holds:
	\begin{align*}
		\mathbb{P}(\{|\langle z^\nu_t \rangle|< \delta \}\cap \Omega^{\gamma})\leq 2C\frac{\delta}{k(\gamma)}\, .
	\end{align*}
	By taking $\gamma=\delta$, the claim follows.
\end{proof}

\subsection{Hausdorff dimension two}
This section is devoted to the proof of the following result: 
\begin{theorem}\label{th:hausdorff}
The measure $\mu$ is at least of two-dimensional nature, in the sense that every compact set of Hausdorff dimension smaller than $2$ has $\mu$-measure $0$.
\end{theorem}
The proof of Theorem \ref{th:hausdorff} relies on two steps:
\begin{itemize}
\item[1.] Let $\mu$ be any limiting point of the family $(\mu^\nu)_\nu$, in the sense of the weak convergence on $H^1(\mathbb{S}^2)$ for $\nu\rightarrow 0$. We define a map $F:(H^1(\mathbb{S}^2), \mathcal{B}_{H^1(\mathbb{S}^2)})\rightarrow (\mathbb{R}^2, \mathcal{B}_{\mathbb{R}^2})$. The push forward measure of $\mu$ under $F$ is continuous with respect to the Lebesgue measure, i.e.~for all $\Gamma \in \mathcal{B}_{\mathbb{R}^2}$
\begin{align}\label{eq:ineq_push_forward}
F_*(\mu)(\Gamma):=\mu(F^{-1}(\Gamma))\leq p(|\Gamma|)\, .
\end{align}
\item[2.] Some measure theory: we adapt the proof of \cite[Corollary 5.2.15]{kuksin_shirikyan_book} to our case.
\begin{corollary}\cite[Corollary 5.2.15]{kuksin_shirikyan_book}
Let $X\in \mathcal{B}_{H^1(\mathbb{S}^2)}$ be a closed subset, whose intersection with any compact set in $H^2(\mathbb{S}^2)$ has finite Hausdorff dimension with respect to the metric on $H^2(\mathbb{S}^2)$ and let $\mu$ be any limiting point of the family $(\mu^\nu)_\nu$, in the sense of the weak convergence on $H^1(\mathbb{S}^2)$ for $\nu\rightarrow 0$. Then $\mu(X)=0$
\end{corollary}
\begin{proof} We recall the proof in \cite[Corollary 5.2.15]{kuksin_shirikyan_book} for the reader's convenience.
We recall that the measure $\mu$ is concentrated on $H^2(\mathbb{S}^2)$. From Ulam's theorem, any Borel measure on a Polish space is regular. Thus, there exists a sequence $(K_n)_n\subset H^2(\mathbb{S}^2)$ of compact sets such that $\lim_{n\rightarrow +\infty}\mu(K_n)=1$. Thus, we prove that $\mu(K_n\cap X)=0$ for all $n\geq 0$.
Let $X\in \mathcal{B}_{H^2(\mathbb{S}^2)}$ be a compact subset of $H^2(\mathbb{S}^2)$ with Hausdorff dimension smaller than two. From the triangle inequality, the map $F$ is uniformly Lipschitz continuous on any compact subset of $H^2(\mathbb{S}^2)$: this implies that $|F(X)|_{\mathbb{R}^2}=0$ (from the definition of Hausdorff dimension). We now set $\Gamma =F(X)$ into \eqref{eq:ineq_push_forward} and conclude that $\mu(X)=0$.
\end{proof}
\end{itemize}

\subsubsection{Proof of Step 1: the main inequality \eqref{eq:ineq_push_forward}}
\begin{theorem}\label{th:hausdorff_intermediate}
	Let $(\mu^\nu)_\nu$ be a sequence of stationary solutions to \eqref{LLG_nu} on $H^1(\mathbb{S}^2)$.
	There is an increasing continuous function $p:\mathbb{R}\rightarrow\mathbb{R}$ such that $p(0)=0$ such that for any limiting point of the family $(\mu^\nu)_\nu$, in the sense of the weak convergence on $H^1(\mathbb{S}^2)$ for $\nu\rightarrow 0$, and for any Borel subset $\Gamma \subset\mathbb{R}^2$, it holds
	\begin{align*}
		F_*(\mu)(\Gamma):=\mu(F^{-1}(\Gamma))\leq p(|\Gamma|)\, ,
	\end{align*}
	where $F(u):=\left(|\langle u\rangle|^2\, , \; \|\partial_x u\|^2_{L^2}\right)$ for $u\in H^1(\mathbb{S}^2)$.
\end{theorem}
The main tool of the proof is the Krylov estimate, which we recall. Consider a $d$-dimensional stationary It\^o process $(y_t)_t$, whose evolution is given by
\begin{align*}
y_t=y_0+\int_0^tx_s\,\mathrm{d}s + \sum_{j\in\mathbb{Z}} \int_0^t \theta^j_s\,\mathrm{d}\beta^j_s\, ,
\end{align*}
where $(\beta^j)_{j\in \mathbb{Z}}$ is a sequence of independent real-valued Brownian motions. We introduce the $d\times d$-dimensional matrix $\sigma$
\begin{align}\label{eq:sigma}
\sigma_{k,\ell}:= \sum_{j\in \mathbb{Z}} \theta^{k, j} \theta^{\ell, j}\, , \quad \mathrm{for} \quad k,\ell \in \{1,\cdots, d\}\, 
\end{align}
and $\theta^{k, j}$ denotes the $j$-th entry of the $d$-dimensional vector $\theta^j$.
\begin{proposition}\label{pro:krylov} [cf. e.g. \cite[Theorem 7.9.1]{kuksin_shirikyan_book}](Krylov's estimate)
Assume the above hypothesis on the stationary process $y$.
There exists a constant $C_d>0$ depending only on the dimension $d\geq 1$, such that for all bounded measurable functions $f:\mathbb{R}^d\rightarrow \mathbb{R}$, the inequality holds for all $t\geq 0$
\begin{align*}
	\mathbb{E}[f(y_t) (\mathrm{det}(\sigma_t))^{1/d}]\leq C_d |f|_d \mathbb{E}[|x_t|]\, .
\end{align*}
\end{proposition}
In the above framework, we study the evolution of the $2$-dimensional process $y_t=F(z^\nu_t)=(|\langle z^\nu_t\rangle|^2, \|\partial_x z^\nu_t\|^2_{L^2})$, where $F:H^2(\mathbb{S}^2)\rightarrow \mathbb{R}^2$. Notice that the quantities appearing in the map $F$ are conservation laws for the SME equation, but not for the stochastic LLG equation. As usual, this fact allows us to extract information on the limit measure $\mu$. We turn to the proof of Theorem \ref{th:hausdorff}
\begin{proof} (\textit{of Theorem \ref{th:hausdorff_intermediate}})
From It\^o's formula, the evolution of the map $F(u)$ reads for all $t\geq 0$
\begin{align*}
F(z^\nu_t)=F(z^\nu_0)+\int_0^t F(z^\nu_r) \dd z^\nu_r+\frac{1}{2} F''(z^\nu_r)\dd \langle \langle z^\nu_r \rangle \rangle \, ,
\end{align*}
where for all $h\in H^2$ the first and second derivatives read
\begin{align*}
F'(z^\nu;h)=\left(2 \langle z^\nu\rangle \cdot \langle h\rangle\, ,\quad 2 \langle  \partial_x z^\nu, \partial_x h\rangle \right) \, ,
\end{align*}
\begin{align*}
F''(z^\nu;h,h)=\left(2 \langle h\rangle \cdot \langle h\rangle\, ,\quad  2 \langle  \partial_x h, \partial_x h\rangle \right) \, .
\end{align*}
To apply the Krylov estimate, we are interested in the integral
\begin{align*}
\int_{0}^{t} F'(z^\nu_r) \sum_j \sqrt{\nu} h_j z^\nu_r\times \dd W^{j}_r\, ,
\end{align*}
a two-dimensional vector whose components are given by
\begin{align*}
\left(\sum_{j\neq 0}\lambda_j \sum_{i\in\{1,2,3\}}\int_{0}^t \sqrt{\nu}  (\langle e_j z^\nu_r\rangle \times \langle z^\nu_r\rangle)^i\dd W^{j,i}_r\, ,\; \sum_{j\neq 0}\lambda_j \sum_{i\in\{1,2,3\}} \int_0^t \sqrt{\nu} \int_D(\partial_x e_j z^\nu_r\times \partial_x z^\nu_r)^i \dd x \dd W^{j,i}_r  \right)\, .
\end{align*}
For every fixed $j\in \mathbb{Z}\setminus \{0\}$, we define a two-dimensional vector $\theta^j$ with first and second component 
\begin{align*}
\theta^{j,1}:= \sqrt{\nu}\lambda_j   (\langle e_j z^\nu_r\rangle \times \langle z^\nu_r\rangle)^i\, ,\quad \theta^{j,2}:= \lambda_j  \sqrt{\nu} \int_D(\partial_x e_j z^\nu_r\times \partial_x z^\nu_r)^i \dd x \, .
\end{align*}
We can now define $\sigma$ as in \eqref{eq:sigma}, and we are in the setting the Krylov inequality. We show that the determinant of $\sigma$ is not singular on a set of positive measure if and only if $\langle z^\nu\rangle \notin \mathbb{S}^2\cup \{0\}$ by proving the counter-nominal. The determinant of $\sigma$ reads 
	\begin{align*}
	&\mathrm{det}(\sigma )=\sum_{j\neq 0} \sum_{i\in\{1,2,3\}}\lambda_j^2\left| \sqrt{\nu} (\langle e_j z^\nu_r\rangle \times \langle z^\nu_r\rangle)^i\right|^2 \sum_{j\neq 0}\sum_{i\in\{1,2,3\}} \lambda_j^2\left|  \sqrt{\nu} \int_D(\partial_x e_j z^\nu_r\times \partial_x z^\nu_r)^i \dd x\right|^2\\
	&\quad-\left|\sum_{j\neq 0} \sum_{i\in\{1,2,3\}}\lambda_j^2\nu  (\langle e_j z^\nu_r\rangle \times \langle z^\nu_r\rangle)^i \int_D(\partial_x e_j z^\nu_r\times \partial_x z^\nu_r)^i \dd x  \right|^2\\
	&=\nu^4\sum_{j,k\neq 0} \lambda_j^2\left|  \langle e_j z^\nu_r\rangle \times \langle z^\nu_r\rangle \right|^2 \lambda_k^2\left|  \int_D\partial_x e_k z^\nu_r\times \partial_x z^\nu_r\dd x\right|^2\\
	&\quad-\nu^4\sum_{j,k\neq 0}\lambda_j^2  \langle e_j z^\nu_r\rangle \times \langle z^\nu_r\rangle\cdot \int_D\partial_x e_j z^\nu_r\times \partial_x z^\nu_r \dd x \lambda_k^2  \langle e_k z^\nu_r\rangle \times \langle z^\nu_r\rangle\cdot \int_D \partial_x e_k z^\nu_r\times \partial_x z^\nu_r \dd x\, .
	\end{align*}
	If $\langle z^\nu\rangle \in \mathbb{S}^2\cup \{0\}$, then $\mathrm{det}(\sigma)=0$.
	To show the other implication, note that
	\begin{align}
	&\quad\sum_{j,k\neq 0, k=j} \lambda_j^2\left|  \langle e_j z^\nu_r\rangle \times \langle z^\nu_r\rangle \right|^2 \lambda_k^2\left|  \int_D \partial_x e_k z^\nu_r\times \partial_x z^\nu_r\dd x \right|^2\nonumber \\
	&= \sum_{j,k\neq 0, k=j} \lambda_j^4\left|  \langle e_j z^\nu_r\rangle \times \langle z^\nu_r\rangle \right|^2 \left| \int_D  \partial_x e_j z^\nu_r\times \partial_x z^\nu_r\dd x\right|^2 \label{eq:det1}\, ,
	\end{align}
	and that for $k=j$ the second summand reads
	\begin{align}
	-\sum_{j,k\neq 0, j=k}\lambda_j^4\left| \langle e_j z^\nu_r\rangle \times \langle z^\nu_r\rangle\cdot \int_D  \partial_x e_j z^\nu_r\times \partial_x z^\nu_r\dd x  \right|^2\label{eq:det2}\, .
	\end{align}
	Thus, summing \eqref{eq:det1} and \eqref{eq:det2}, 
	\begin{align*}
		\sum_{j\neq 0}\lambda_j^4 \left|(\langle e_j z^\nu_r\rangle \times \langle z^\nu_r\rangle)\times \int_D  \partial_x e_j z^\nu_r\times \partial_x z^\nu_r\dd x \right|^2\, .
	\end{align*}
	Since 
	\begin{align*}
	&\quad-\sum_{j,k\neq 0, j\neq k}\lambda_j^2  (\langle e_j z^\nu_r\rangle \times \langle z^\nu_r\rangle)\cdot \int_D  \partial_x e_j z^\nu_r\times \partial_x z^\nu_r\dd x \lambda_k^2  (\langle e_k z^\nu_r\rangle \times \langle z^\nu_r\rangle)\cdot\int_D \partial_x e_k z^\nu_r\times \partial_x z^\nu_r\dd x\\
	&\geq -\sum_{j,k\neq 0, j\neq k}\lambda_j^2 \lambda_k^2  |\langle e_j z^\nu_r\rangle \times \langle z^\nu_r\rangle|\left|\int_D\partial_x e_j z^\nu_r\times \partial_x z^\nu_r\dd x\right||\langle e_k z^\nu_r\rangle \times \langle z^\nu_r\rangle|\left|\int_D\partial_x e_k z^\nu_r\times \partial_x z^\nu_r\dd x\right|\\
	&\geq -\sum_{j,k\neq 0, j\neq k} \frac{\lambda_j^2 \lambda_k^2}{2} |\langle e_j z^\nu_r\rangle \times \langle z^\nu_r\rangle|^2\left|\int_D\partial_x e_k z^\nu_r\times \partial_x z^\nu_r\dd x\right|^2\\
	&\quad-\sum_{j,k\neq 0, j\neq k} \frac{\lambda_j^2 \lambda_k^2}{2} |\langle e_k z^\nu_r\rangle \times \langle z^\nu_r\rangle|^2\left|\int_D\partial_x e_j z^\nu_r\times \partial_x z^\nu_r\dd x \right|^2 \, ,
	\end{align*}
	we observe that 
	\begin{align*}
		\mathrm{det} (\sigma)\geq \sum_{j\neq 0}\lambda_j^4 \left|(\langle e_j z^\nu_r\rangle \times \langle z^\nu_r\rangle)\times \int_D  \partial_x e_j z^\nu_r\times \partial_x z^\nu_r\dd x \right|^2\, .
	\end{align*}
	Observe that  $\mathrm{det}(\sigma)>0$ $\mathbb{P}$-a.s., provided there exists $j\in \mathbb{Z}\setminus \{0\}$ such that for all $\alpha, \beta \neq 0$
	\begin{align}\label{eq:det3}
		\alpha (\langle e_j z^\nu_r\rangle \times \langle z^\nu_r\rangle) +\beta \int_D\partial_x e_j z^\nu_r\times \partial_x z^\nu_r \dd x  \neq 0\, .
	\end{align}
	\textit{Proof that $\mathrm{det}(\sigma)=0$ $\mathbb{P}$-a.s.~implies that $\langle z^\nu\rangle \in \mathbb{S}^2\cup \{0\}$.}
	We argue by contradiction: assume that there exist $(\alpha, \beta )\neq (0,0)$ so that for all $j\in \mathbb{Z}\setminus \{0\}$, the vector in \eqref{eq:det3} is the zero vector. Then for all $i,j\in \mathbb{Z} \setminus \{0\}$,
	\begin{align*}
	\alpha (\langle e_j z^\nu_r\rangle \times \langle z^\nu_r\rangle)^i +\beta \int_D(\partial_x e_j z^\nu_r\times \partial_x z^\nu_r)^i \dd x = 0\, .
	\end{align*}
	If $\alpha =0, \beta \neq0 $, then $z^\nu\times \partial^2_x z^\nu\equiv 0$ and $z^\nu=\langle z^\nu\rangle$ from the Poincaré-Wirtinger inequality. If $\alpha \neq 0, \beta =0 $, for all $i,j\in \mathbb{Z} \setminus \{0\}$,
	\begin{align*}
	\alpha (\langle e_j z^\nu_r\rangle \times \langle z^\nu_r\rangle)^i = 0\quad \implies\quad   e_j z^\nu_r \times \langle z^\nu_r\rangle=0\, .
	\end{align*}
	Hence $z^\nu_r \times \langle z^\nu_r\rangle=0$, which implies that either $z^\nu_r=\langle z^\nu_r\rangle$ or $\langle z^\nu_r\rangle=0$.
	Assume now that $\alpha \neq 0, \beta \neq0 $. This implies that for all $j\in \mathbb{Z}\setminus \{0\}$
	\begin{align}\label{eq:eq4}
	\alpha z^\nu_r \times \langle z^\nu_r\rangle- \beta  z^\nu_r\times \partial^2_x z^\nu_r= z^\nu_r \times (\alpha \langle z^\nu_r \rangle -\beta \partial^2_x z^\nu_r )=0\, .
	\end{align}
	The above relation holds provided either $\alpha \langle z^\nu_r \rangle -\beta \partial^2_x z^\nu_r =0$ or $z^\nu_r=\gamma (\alpha \langle z^\nu_r \rangle -\beta \partial^2_x z^\nu_r)$ for $\gamma \in \mathbb{R}\setminus \{0\}$.
	Assume that $\alpha \langle z^\nu_r \rangle -\beta \partial^2_x z^\nu_r =0$ and integrate in space: it follows that $\langle z^\nu_r\rangle=0$.	
	
	 We assume now that $z^\nu_r=\gamma (\alpha \langle z^\nu_r \rangle -\beta \partial^2_x z^\nu_r)$ for $\gamma \in \mathbb{R}\setminus \{0\}$. We reduce the number of parameters by dividing by $\beta\in \mathbb{R}\setminus \{0\}$ in \eqref{eq:eq4} and considering
	 \begin{align} \label{eq:5a}
	  z^\nu_r \times \langle z^\nu_r \rangle -\beta z^\nu_r \times\partial^2_x z^\nu_r =0\, , 
	  \end{align}
	  \begin{align} \label{eq:R1}
	 z^\nu_r=\gamma  \langle z^\nu_r \rangle -\gamma \beta \partial^2_x z^\nu_r   \, . 
	 \end{align}
	 From the second relation, by integrating in space, we deduce that $\gamma=|D|$. We can project now $z^\nu$ on the basis elements and for all $j\neq 0$
	 \begin{align*}
	 \langle 	z^\nu_r , e_j \rangle =\gamma \langle  \langle z^\nu_r \rangle , e_j \rangle-\gamma \beta \langle  \partial^2_x z^\nu_r  , e_j \rangle= \gamma j^2 \beta \langle z^\nu_r  , e_j\rangle\, .
	 \end{align*}
	 Since $\beta\in \mathbb{R}$ is a fixed value for all $j$, we conclude that $\langle 	z^\nu_r , e_k \rangle =0$ for all $j\neq k, -k\, , j\in \mathbb{Z}\setminus \{0\}$;  for $j=k, -k $, $\langle 	z^\nu_r , e_j \rangle \neq 0$ and $\beta =1/|D|k^2$. 
	 Assume that $z^\nu_r=\langle 	z^\nu_r , e_k \rangle e_k$, i.e.~$	z^\nu_r$ is a one-dimensional process. Then 
	 \begin{align*}
	 	z^\nu_r\times \partial_x^2 z^\nu_r= \langle 	z^\nu_r , e_k \rangle e_k\times \langle 	z^\nu_r , e_k \rangle e_k k^2=0\, .
	 \end{align*}
	 From the conserved relation $\mathbb{E}[\|z^\nu_r\times \partial_x^2 z^\nu_r\|^2_{L^2}]=L^2\neq 0$, this leads to a contradiction. Assume that 
	 \begin{align}\label{eq:R2}
	 z^\nu_r=\langle 	z^\nu_r \rangle+\langle 	z^\nu_r , e_k \rangle e_k+\langle 	z^\nu_r , e_{-k} \rangle e_{-k}\, ,
	 \end{align}
	 then 
	 \begin{align*}
	 	\partial_x z^\nu_r =\langle 	z^\nu_r , e_k \rangle \partial_x e_k+\langle 	z^\nu_r , e_{-k} \rangle \partial_x e_{-k} =\langle 	z^\nu_r , e_k \rangle  k e_{-k}-k\langle 	z^\nu_r , e_{-k} \rangle e_k \, .
	 \end{align*}
	 \begin{align}\label{eq:R2_not}
	 	\partial^2_x z^\nu_r = -\langle 	z^\nu_r , e_k \rangle  k^2 e_{k}-\langle 	z^\nu_r , e_{-k} \rangle k^2 e_{-k}\, .
	 \end{align}
    
	 On the other hand, from \eqref{eq:R1} and since $1/\gamma\beta=k^2$, 
	 \begin{align}\label{eq:other_laplacian}
	 	\partial_x^2 z^\nu_r&= k^2 (|D|\langle z^\nu_r \rangle-\langle 	z^\nu_r \rangle-\langle 	z^\nu_r , e_k \rangle e_k-\langle 	z^\nu_r , e_{-k} \rangle e_{-k})\, .
	 \end{align}
	 the last equality follows from \eqref{eq:R2}.
	 Comparing \eqref{eq:R2_not} and \eqref{eq:other_laplacian}, it follows that $\langle z^\nu_r\rangle=0$.
	 \\
	 
	 \textit{Conclusion:}  Define now 
	 \begin{align*}
	 J^\epsilon:=\{u\in H^2_0(\mathbb{S}^2):\epsilon\leq|\langle u\rangle |^2\leq 1-\epsilon, \|\partial^2_x u\|^2_{L^2}\leq1/\epsilon \}\, ,
	 \end{align*}
	 which is bounded in $H^2_0(\mathbb{S}^2)$ and therefore compact in $H^1_0(\mathbb{S}^2)$. The set $J^\epsilon$ is separated from $0$: the space average can be neither $0$ nor $1$ and the derivative $\|\partial_x u\|^2_{L^2}$ and $\|\partial_x u\|^2_{L^2}$ cannot vanish from the Poincaré-Wirtinger inequality.
	  The map $\det(\sigma)$ is continuous from $H^1_0(\mathbb{S}^2)$ to $\mathbb{R}$; indeed for all $f\in H^1_0(\mathbb{S}^2)$,
	 \begin{align*}
	 	\mathrm{det}(\sigma(f) )&=\sum_{j\neq 0} \sum_{i\in\{1,2,3\}}\lambda_j^2\left| \sqrt{\nu} (\langle e_j f\rangle \times \langle f\rangle)^i\right|^2 \sum_{j\neq 0}\sum_{i\in\{1,2,3\}} \lambda_j^2\left|  \sqrt{\nu} \int_D(\partial_x e_j f \times \partial_x f)^i \dd x\right|^2\\
	 	&\quad-\left|\sum_{j\neq 0} \sum_{i\in\{1,2,3\}}\lambda_j^2\nu  (\langle e_j f \rangle \times \langle f \rangle)^i \int_D(\partial_x e_j f\times \partial_x f )^i \dd x  \right|^2\\
	 	&\lesssim \sum_{j \neq 0}\lambda_j^2 \lambda_1^2 \|\partial_x f\|^2_{L^2}=L\lambda_1^2\|\partial_x f\|^2_{H^1_0(\mathbb{S}^2)} \, .
	 \end{align*}
	 From the continuity of the determinant and since $J^\epsilon$ is compact in $H^1_0(\mathbb{S}^2)$ and separated from $0$, there exists $c_\epsilon>0$ such that $\mathrm{det}(J^\epsilon\cap \Gamma)>c_\epsilon$ and $\lim_{\epsilon\rightarrow 0} c_\epsilon =0$.
	 The continuity of the law of $F(z^\nu)$ with respect to the Lebesgue measure follows from the Krylov estimate and Proposition \ref{pro:space_av_small}. Indeed 
	 \begin{align*}
	 	\mu^\nu (F(u)\in \Gamma)=\mu^\nu (\{F(u)\in \Gamma\}\cap J^\epsilon )+\mu^\nu (\{F(u)\in \Gamma\}\cap J^{\epsilon,C})\leq C c_\epsilon^{-1}|\Gamma|_{\mathbb{R}^2}+ 2\epsilon\, .
	 \end{align*}
	 By choosing $\epsilon$ such that $|\Gamma|_{\mathbb{R}^2}/c_\epsilon =p(|\Gamma|_{\mathbb{R}^2})$ and from Portmanteau's theorem, the assertion follows. 
	
\end{proof}

\subsection{Proof of Theorem \ref{th:mu_tilde_BCF} d.}\label{sec:proof_BCF}
From Theorem \ref{th:mu} d.~and the equalities in \eqref{eq:rel_SME_BCF}, 
\begin{align*}
	\mu (F(u)\in \Gamma) &=\mathbb{P}(\{\omega\in \Omega: (|\langle z \rangle|^2\, ,\; \|\partial_x z_r\|^2_{L^2}) \in \Gamma \})\\
	&= \mathbb{P}\left(\{\omega\in \Omega: \left(\frac{|v(2\pi)|^2}{4\pi^2}\, ,\; \|\partial^2_x v_r\|^2_{L^2}\right) \in \Gamma \}\right)\\
	&=\tilde{\mu}\left(\{u\in H^3: \left(\frac{|u(2\pi)|^2}{4\pi^2}\, ,\; \|\partial^2_x u_r\|^2_{L^2}\right) \in \Gamma\}\right) \leq p(|\Gamma|_{\mathbb{R}^2})\, .
\end{align*}

\section{Declarations}
\subsection{Statement and declarations:} The authors have no relevant financial or non-financial interests to disclose.

\subsection{Author contribution:} The authors assume full responsibility for the study's conception, the results presented, and the preparation of the manuscript.

\subsection{Funding declaration:} E.G.~gratefully acknowledges the financial support of the Deutsche Forschungsgemeinschaft (DFG, German Research Foundation) – SFB 1283/2 2021 – 317210226 (Project B7). M.S.~is grateful for the research of M. Sy is funded by the Alexander von Humboldt foundation under the “German Research Chair programme” financed by the Federal Ministry of Education and Research (BMBF).

\appendix
\section{Known preliminary results.}
This section recalls some known results.
\begin{theorem}(\cite[Theorem 2.2]{Shirikyan_local_times}, \cite[Theorem~7.8.1]{kuksin_shirikyan_book}) \label{th:local_times_appendix} Let $(y_t)_t$ be a  \textit{ standard Itô process} of the form 
	\begin{align*}
		y_t=y_0+\int_0^tx_s\,\mathrm{d}s + \sum_{n\in\mathbb{Z}} \int_0^t \theta^n_s\,\mathrm{d}\beta^n_s\,,
	\end{align*}
	where $(x_t)_t, (\theta^n_t)_t$ are $\mathcal{F}_t$-adapted processes such that there holds \[\mathbb{E} \left[\int_0^t \left(|x_s|+\sum_{n \in\mathbb{Z}} |\theta^n_s|^2\right)\,\mathrm{d}s\right] < \infty \text{ for any } t>0\,.\]
	Then there exists a random field that we denote by $\Lambda_t(a,\omega)$, for $t \geqslant 0$, $a \in \mathbb{R}$, $\omega \in \Omega$ such that the following properties hold. 
	\begin{enumerate}[label=(\textit{\roman*})]
		\item $(t,a,\omega) \mapsto \Lambda_t(a,\omega)$ is measurable and for any $a \in \mathbb{R}$ the process $t \mapsto \Lambda_t(a, \cdot)$ is $\mathcal{F}_t$-adapted continuous and non-decreasing. For any $t\geqslant 0$, and almost every $\omega \in \Omega$ the function $a \mapsto \Lambda _t (a, \omega)$ is right-continuous. 
		\item For any non-negative Borel function $g : \mathbb{R} \to \mathbb{R}$ and with probability $1$, we have for any $t \geqslant 0$
		\begin{equation}
		\label{4eqlocal1}
		\int_0^tg(y_s) \left(\sum_{n \in\mathbb{Z}} |\theta^n_s|^2\right)\,\mathrm{d}s = 2 \int_{\mathbb{R}}\Lambda_t(a,\omega)\,\mathrm{d}a\,.
		\end{equation}
		\item For any convex function $f : \mathbb{R} \to \mathbb{R}$ and with probability $1$, it holds 
        \begin{align}
		\label{4eqlocal2}
		f(y_t)&=f(y_0)+ \int_0^t \partial ^{-}f(y_s)x_s\,\mathrm{d}s + \int_{\mathbb{R}} \Lambda _t (a, \omega) \partial^2f(\mathrm{d}a)\\
		& + \sum_{n \in\mathbb{Z}} \int_0^t \partial^{-}f(y_s) \theta^n_s \,\mathrm{d}\beta^n _s\,,\notag
		\end{align}
		and where $\partial ^-f$ denotes the subdifferential of $f$. 
	\end{enumerate}
\end{theorem}

\bibliographystyle{plain}

\begin{thebibliography}{9}

    \bibitem{banica_luca_tzv_vega}
    V.~Banica, R.~Lucà, N.~Tzvetkov, L.~Vega.
    On low regularity well-posedness of the binormal flow. ArXiv:arXiv:2505.23284, 2025.
    
	\bibitem{banica_vega_1}
	V. Banica and L. Vega.
	Stability of the self similar dynamics of a
	vortex filament, Arch. Ration.
	Mech. Anal. 210 (2013), 673–712.
	
	\bibitem{banica_vega_2}
	V. Banica and L. Vega.
	Evolution of polygonal lines by the binormal
	flow, Ann. PDE, 6,
	Paper No. 6, pp. 53, (2020).
	
	\bibitem{banica_vega_3}
	V. Banica and L. Vega, On the energy of critical solutions of the
	binormal flow, Comm. PDE, 45
	(2020), 820–845 .
	
	\bibitem{banica_vega_4}
	Banica, V., Vega, L. Riemann's Non-differentiable Function and the
	Binormal Curvature Flow. Arch. Ration. Mech. Anal., 244, 501–540
	(2022) .
	

    \bibitem{bess_gub_russo}
    Bessaih, H., Gubinelli, M., Russo, F., 2005. The evolution of a random vortex filament. Ann. Probab. 33, 1825–1855.
	
	\bibitem{Bourgain}
	J.~Bourgain. Invariant measures for the 2D-defocusing nonlinear Schr\"dinger
	equation.
	\textit{Comm. Math. Phys.} 176: 421-455, 1996.
	
	\bibitem{bourgain_94}
	J. Bourgain. 
	Periodic nonlinear Schr\"odinger equation and invariant measures. Comm. Math. Phys., 166(1):1–26, 1994.
	
	
	\bibitem{martina_book}
	D.~Breit, E.~Fereisl, M.~Hofmanová.
	Stochastically forced compressible fluid flows.
	Series in Applied and Numerical Mathematics, De Gruyter (2018).

	\bibitem{brzezniak_LDP}
	Z.~Brze\'zniak, B.~Goldys, T.~Jegaraj.
	Large deviations and transitions between equilibria for stochastic Landau-Lifshitz-Gilbert equation.
	\textit{Archive for Rational Mechanics and Analysis.} 1--62, 2017.
	
	\bibitem{brz_gol_li}
	Z.~Brze\'zniak, B.~Goldys, L.~Li.
	3D Stochastic Landau-Lifshitz-Gilbert Equations coupled with Maxwell's Equations with full energy.
	Journal of Differential Equations 390, 58-124, 2024.

    \bibitem{brz_gub_nek}
    Z Brzeźniak, M Gubinelli and M Neklyudov.
    Global solutions of the random vortex filament equation.
    2013 Nonlinearity 26 2499

    \bibitem{burq_thomann_tzvetkov}
    N.~Burq, L.~Thomann, N.~Tzvetkov
    Remarks on the Gibbs measures for nonlinear dispersive equations. Annales de la Faculté des sciences de Toulouse, 2018.

	
	\bibitem{da_rios}
	L.~S.~Da Rios, On the motion of an unbounded fluid with a vortex filament of any shape, Rend. Circ. Mat. Palermo 22 (1906), 117–135.

	\bibitem{fahim_h}
	K.~Fahim, E.~Hausenblas, D.~Mukherjee. 
	Wong–Zakai Approximation for Landau–Lifshitz–Gilbert Equation Driven by Geometric Rough Path. {\em Applied mathematics and optimization}, \textbf{46}, 1685-1730, 2021.

	\bibitem{flandoli_gubinelli_1}
	 F.~Flandoli, M.~Gubinelli, 
	 The Gibbs ensemble of a vortex filament, Probab Theory Relat Fields 122 (2002), 317–340.
	 
	 \bibitem{flandoli_gubinelli_2}
	 F.~Flandoli, M.~Gubinelli, 
	 Statistics of a Vortex Filament Model.
	 Electron. J. Probab, 10 (2005), 865–900.
	
	\bibitem{foldes_sy}
	J.~F\"oldes and M.~Sy. Invariant measures and global well posedness for SQG equation. arXiv:2002.09555, 2020.

    \bibitem{foldessy2023almost}
    Foldes, J. and Sy, M.
    Almost sure global well-posedness for {3D} {E}uler equation and other fluid dynamics models.
    arXiv preprint arXiv:2401.00332, 2023.
	
	\bibitem{FrizHairer}
	P.~Friz, M.~Hairer.
	\textit{A course on rough paths: with an introduction to regularity structures}.
	Universitext, Springer, 2014.
	
	\bibitem{stochastic_LL}
	B.~Guo, X.~Pu.
	Stochastic Landau-Lifschitz equation. Differential Integral Equations 22 (3/4) 251 - 274, 2009.
	
	\bibitem{LLG_inv_measure}
	E.~Gussetti.
	On ergodic invariant measures for the stochastic Landau-Lifschitz-Gilbert equation in 1D. Arxiv preprint (2022).
	
	\bibitem{CLT}
	E.~Gussetti.
	Pathwise central limit theorem and moderate deviations via rough paths for SPDEs with multiplicative noise.
	Electron. J. Probab. 30: 1-51 (2025). 
	
	\bibitem{LLG1D}
	E.~Gussetti, A.~Hocquet.
	A pathwise stochastic Landau-Lifshitz-Gilbert equation with application to large deviations. Journal of Functional Analysis, 285 (9), 110094, 2023.
	
	\bibitem{SME}
	E.~Gussetti, M.~Hofmanová.
	statistically solutions to the Schr\"odinger map equation in 1D, via the randomly forced Landau-Lifschitz-Gilbert equation. arXiv arXiv:2501.16499, 2025.

	\bibitem{holtz_sverak}
	N. Glatt-Holtz, V. Sverak, and V. Vicol. On inviscid limits for the stochastic Navier–Stokes equations and related models. Arch. Ration. Mech. Anal., 217:619–649, 2015.

    \bibitem{hocquet_neamtu}
    A.~Hocquet, A.~Neamţu. Quasilinear rough evolution equations. {\em The Annals of Applied Probability},  34 (5), 4268-4309, 2024.

	\bibitem{kuksin_2003}
	S.~Kuksin.
	The Eulerian Limit for 2D statistically Hydrodynamics.
	Journal of statistically Physics, Vol. 115, Nos. 1/2, 2004.
	
	\bibitem{kuksin_2008}
	S.~Kuksin.
	On Distribution of Energy and Vorticity for Solutions of 2D Navier-Stokes Equation with Small Viscosity. Commun. Math. Phys. 284, 407 (2008). 

	\bibitem{kuksin_shirikyan_2004}
	S.~Kuksin, A.~Shirikyan.
	Randomly forced CGL equation: stationary measures and the inviscid limit. J. Phys. A: Math. Gen. \textbf{37} 3805, 2004.
	
	\bibitem{kuksin_shirikyan_book}
	S.~Kuksin, A.~Shirikyan.
	Mathematics of Two-Dimensional Turbulence. Vol. 194. Cambridge University Press, 2012.
	
	\bibitem{latocca}
	M.~Latocca. Construction of high regularity invariant measures for the 2D Euler equations and remarks on the growth of the solutions. Communications in Partial Differential Equations, 2023, 48 (1), pp.22-53.
	
	\bibitem{lebowitz_rose_speer}
	J. L. Lebowitz, H. A. Rose, and E. R. Speer. statistically mechanics of the nonlinear Schr\"odinger equation. J. Statist. Phys., 50(3-4):657–687, 1988.

    \bibitem{neklyudov2013role}
    M.~Neklyudov and A.~Prohl. 
    The role of noise in finite ensembles of nanomagnetic particles. \textit{Archive for Rational Mechanics and Analysis} \textbf{210}(2): 499--534, 2013.

	\bibitem{protter}
	P.~E.~Protter. 
	Stochastic integration and differential equations. (2005)
	
	
	\bibitem{Shirikyan_local_times}
	A.~Shirikyan.
	Local times for solutions of the complex Ginzburg–Landau equation and
	the inviscid limit. Journal of Mathematical Analysis and Applications,
	
	
	
	\bibitem{sy_xu_2021}
	M.~Sy, X.~Yu.
	Almost sure global well-posedness for the energy supercritical NLS on the unit ball of $\mathbb{R}^3$. Preprint arXiv:2007.00766.
	
	
	\bibitem{sy_b_ono}
	M.~Sy.
	Invariant measure and long time behavior of regular solutions of the Benjamin-Ono equation. Anal. PDE 11(8): 1841-1879 (2018).
	
	
	
	\bibitem{sy_19}
	M.~Sy.
	Almost sure global well-posedness for the energy supercritical Schrödinger equations.  J. Math. Pures. Appl. (154) 108--145, 2019.

	\bibitem{sy_xu_2021_2}
	M.~Sy, X.~Yu.
	Global well-posedness and long-time behavior of the fractional NLS.
	Stochastics and Partial Differential Equations: Analysis and Computations (10), 1261–1317, 2021.

    \bibitem{Tzvetkov_2008}
    N.~Tzvetkov
    Invariant measures for the defocusing nonlinear Schrödinger equation. Annales de l'Institut Fourier, (27) 3, 2543-2604,  2008.

    \bibitem{Zakharov_Shabat}
    V.~E.~Zakharov and A.~B.~Shabat
    Interaction between solitons in a stable medium.
    Zh. Eksp. Tear. Fiz 64,1627-1639 (May 1973)

    \bibitem{norbs}
    A.~R.~Nahmod, T.~Oh, L.~Rey-Bellet, G.~Staffilani
    Invariant weighted Wiener measures and almost sure global well-posedness for the periodic derivative NLS.
     Journal of the European Mathematical Society, 2012.
	
\end{thebibliography}

\end{document}